\documentclass{amsart}

\usepackage[dvipsnames]{xcolor}

\usepackage[
colorlinks=true,    % Removes the boxes and colors the text instead
linkcolor=NavyBlue,     % Color for \ref, \eqref, \cref, and table of contents
citecolor=ForestGreen,    % Color for \cite
urlcolor=blue,      % Color for external URLs
breaklinks=true     % Allows links to wrap over line breaks
]{hyperref}
\usepackage{hyperref-patches}

\usepackage{enumitem}
\usepackage{mathrsfs}
\usepackage{float}
\usepackage{bm}
\usepackage{amsrefs}
\usepackage{esint}

\usepackage{cleveref}
\newtheorem{theorem}{Theorem}[section]
\newtheorem*{theorem*}{Theorem}
\newtheorem{lemma}[theorem]{Lemma}
\newtheorem{proposition}[theorem]{Proposition}
\newtheorem{cor}[theorem]{Corollary}
\usepackage[T1]{fontenc}
\usepackage{times}
\usepackage{bbold}

\crefname{lemma}{lemma}{lemmas}
\crefname{theorem}{theorem}{theorems}
\crefname{proposition}{proposition}{propositions}
\crefname{cor}{corollary}{corollaries}
\Crefname{lemma}{Lemma}{Lemmas}
\Crefname{theorem}{Theorem}{Theorems}
\Crefname{proposition}{Proposition}{Propositions}
\Crefname{cor}{Corollary}{Corollaries}

\theoremstyle{definition}

\newtheorem*{remark}{Remark}
\newtheorem*{claim}{Claim}
\numberwithin{equation}{section}

\newcommand{\N}{\mathbb{N}}
\newcommand{\R}{\mathbb{R}}

\renewcommand{\P}{\mathbb{P}}

\renewcommand{\S}{\mathcal{S}}
\renewcommand{\L}{\mathscr{L}}
\newcommand{\density}{\delta}
\newcommand{\mtc}{\eta}

\DeclareMathOperator{\Vol}{Vol}

\begin{document}
	
	\title[Multiplication tables with restricted prime factors]{Multiplication tables for integers with restricted prime factors}
	
	\author{Jeremy Schlitt}
	\address{D\'epartement de Math\'ematiques et de Statistique, Universit\'e de Montr\'eal, CP 6128 succ. centre, Centre Ville, Montr\'eal, Qc H3T 1J4, CANADA}
	\thanks{The author was supported in part by the Bourse d'Excellence - Banque Nationale.}
	

	\date{\today}

	
	\begin{abstract}
		Let $Q$ be a set of primes with relative density $\density$. We count integers in $[1,x]$ with prime factors all in $Q$ that also have a divisor in $(y,2y]$. We establish the order of magnitude for all $\density \in (0,1]$. This generalizes the case $\density = 1$ from the 2008 work of Ford. We also show that there is a phase transition at the critical point $\density = 1/\log 4$, for which we explicitly determine the behaviour. Along the way, we prove estimates for uniform order statistics with strong barrier conditions which are of independent interest.
	\end{abstract}
	\maketitle

	%% The correct journal style for \specialsection is all uppercase; a known bug
	%% in amsart.cls prevents this, so input must be uppercase until it is fixed.
	%\specialsection*{This is a Special Section Head}
	
	%%%%%%%%%%%%%%%%%%%%%%%%%%%%%%%%%%%%%%%%%%%%%%%%%%%%%%%%%%%%%%%%%%%%%%%%
	
	%%%%%%%%%%%%%%%%%%%%%%%%%%%%%%%%%%%%%%%%%%%%%%%%%%%%%%%%%%%%%%%%%%%%%%%%
	
	\section{Introduction}
	In his seminal 2008 work \cite{ford2008distribution}, Ford established the correct order of magnitude for the number $H(x,y,z)$ of integers not exceeding $x$ with a divisor in the interval $(y,z]$. For example, one consequence of his work is:
	\begin{theorem}[Ford \cite{ford2008distribution}, 2008] Suppose $100 \leq y \leq \sqrt{x}$. Then\label{fordthm}
		\begin{equation}\label{fordmain H}
			H(x,y,2y) \asymp \frac{x}{(\log y)^\mtc(\log \log y)^{3/2}},
		\end{equation}
		where 
		$$\mtc:=1-\frac{1+\log \log 2}{\log 2}.$$
	\end{theorem}
	The number $\mtc$ is called the \textit{multiplication table constant}. The case $z =2y$ mentioned above is of particular interest, as it allows one to provide the correct order of magnitude to Erd\H{o}s' \textit{Multiplication Table Problem} \cite{Erdos1960}. Indeed, let $A(N)$ denote the number of distinct integers in the $N \times N$ multiplication table, that is to say
	$$A(N) := \#\{ab: a,b\leq N\}.$$
	Using \Cref{fordthm}, it can be shown \cite[Corollary 3]{ford2008distribution} that $$A(N) \asymp \frac{N^2}{(\log N)^\mtc(\log \log N)^{3/2}}.$$
	
	The aim of the present paper is to study multiplication tables where the row and column numbers are restricted to a special set which is of arithmetic interest. To be more precise, we start with an infinite set $Q$ of prime numbers for which we suppose there exists ${\density \in (0,1]}$ and $\kappa \ge 2$ such that
	
	\begin{equation}\label{definition of delta}
		\Big|\#\big(Q\cap[1,x]\big) - \frac{ \delta x}{\log x}\Big|\le \frac{\kappa x}{(\log x)^2}, \quad ( x\ge2 ).
	\end{equation}
	
	 From here, we define the set $$\S_Q:= \{n \in \N : p|n \implies p \in Q\}.$$
	We will frequently write $\S$ in place of $\S_Q$ for ease of notation. We now define
	\begin{equation*}
		H_Q(x,y,z) = \#\{n \in \S_Q\cap[1,x]: \exists d|n\text{ s.t. }d\in (y,z]\}.
	\end{equation*}
	In the present paper, we will focus on the important case $z=2y$. When $Q$ is the set of all primes ($\density = 1$) and $z=2y$, we recover the usual $H(x,y,2y)$ of \eqref{fordmain H}. We now state our main theorem. 
	\begin{theorem} \label{mainthm}
		Let $c \in (0,1/10]$ and $\kappa \ge 2$ be given. Let $y_0=y_0(c,\kappa)$ be a sufficiently large number in terms of $c,\kappa$. Suppose that ${y_0 \leq y \leq \sqrt{x}}$. One has, uniformly over all sets of primes $Q$ satisfying \eqref{definition of delta} with $\delta \ge c$, that
		\begin{equation}
			H_Q(x,y,2y) \asymp_{\kappa,c} \frac{x}{(\log x)^{1-\density}} \cdot \frac{1}{(\log y)^{G(\density)}} \cdot E(y;\density),
		\end{equation}
		where $$G(\density) := \begin{cases}
			1-\density, \quad &\density \leq \frac{1}{\log 4};\\
			\noalign{\vskip-12pt}\\
			\density -  \frac{1+\log(\density \log 2)}{\log 2}, & \density \geq \frac{1}{\log 4},
		\end{cases}$$
		and
		$$E(y;\density) := \begin{cases}
			\max\big\{\frac{1}{\log 4} - \density, \frac{1}{\sqrt{\log \log y}}\big\} & \text{if } \density \leq \frac{1}{\log 4}, \\
			\noalign{\vskip 5pt}
			\frac{1}{(\log \log y)^{3/2} \max\big\{(\density - \frac{1}{\log 4})^2, \frac{1}{\log \log y}\big\}} & \text{if } \density > \frac{1}{\log 4}.
		\end{cases}$$
	\end{theorem}
		One can verify that the function $G(\delta)$ is differentiable at $\delta = 1/\log 4$, and so there is a continuity in the power of $\log y$ in \Cref{mainthm} as we vary $\density$. Furthermore, the function $E(y;\density)$ gives us the explicit transition behavior of $H_Q(x,y,2y)$ about the point $\density = 1/\log 4$; $E(y;\density)$ varies in a continuous way as $\density$ transitions from the left of $1/\log 4$ to the right.
		
	Using \Cref{mainthm}, we can obtain an estimate on the number of distinct integers in the multiplication table restricted to integers in $\S_Q \cap [1,N]$. Define
	$$A_Q(N) := \#\{ab: a,b \in \S_Q\cap[1,N]\}.$$
	We have the relations\footnote{See Theorem 23 of \cite{hall1988divisors} or the proof of Corollary 3 of \cite{ford2008distribution}.}
	\begin{equation}
		H_Q\left(\frac{N^2}{4},\frac{N}{4},\frac{N}{2} \right) \leq A_Q(N)  \leq \sum_{k \geq 0} H_Q\left(\frac{N^2}{2^k},\frac{N}{2^{k+1}},\frac{N}{2^k} \right),
	\end{equation}
	which immediately lead to the following corollary.
	\begin{cor} \label{maincor} Let $c,\kappa,y_0$, $G(\density),E(y;\density)$ be as in \Cref{mainthm}. Suppose $N > y_0$. Then uniformly over all sets of primes $Q$ satisfying \eqref{definition of delta} with $\delta \ge c$, one has
		\begin{equation}
		A_Q(N) \asymp_{c,\kappa} \frac{N^2}{(\log N)^{G(\density) - \density +1}}E(N;\density).
	\end{equation}
	\end{cor}
	\Cref{maincor} has the following consequence. Let $\S_Q(N) = \S_Q \cap [1,N]$ and note that $A_Q(N) = |\S_Q(N) \cdot \S_Q(N)|$, where $X \cdot Y$ denotes the set of multiples $xy$ with $x \in X, y\in Y$. We obviously have that $|\S_Q(N) \cdot \S_Q(N)| \le |\S_Q(N)|^2$. By \Cref{count s P-} and \Cref{maincor}, it follows that for any fixed set of primes $Q$ (and fixed $\delta$) satisfying \eqref{definition of delta} and $N \to \infty$ we have
	\begin{equation}\label{comb inter}
		 \begin{cases}
		|\S_Q(N) \cdot \S_Q(N)| \asymp	|\S_Q(N)|^2 \quad &\text{if } \density< 1/\log 4;\\|\S_Q(N) \cdot \S_Q(N)| = o(|\S_Q(N)|^2) \quad &\text{if } \density \geq 1/\log 4,
		\end{cases}
	\end{equation}
	and in fact we have the correct order of magnitude of $|\S_Q(N) \cdot \S_Q(N)|$ uniformly for all $\delta \in (0,1]$. This combinatorial interpretation of our theorem has clear connections to two recent papers: one of Soundararajan and Xu \cite{soundararajan2023central}, and one of Ford \cite{ford2018extremal}. We explain these connections presently.
	
	Corollary 1.3 of \cite{soundararajan2023central} states that when one takes $\mathcal{A}$ to be the set of all sums of two squares in the short interval $[x,x+y]$ with $x^{1/3}< y = o(x)$, then as $f$ varies over Steinhaus random multiplicative functions, the quantity
	$$\frac{1}{\sqrt{|\mathcal{A}|}}\sum_{n \in \mathcal{A}}f(n)$$
	is distributed like a standard complex normal random variable. At the end of Section 5 in their paper, it is remarked that the result will also hold true if one replaces the set $\mathcal{
	A}$ with (in the notation of the present paper) $\S_Q \cap [x,x+y]$, so long as $\delta < 1/\log 4 -\epsilon$ and $y$ is not too small. The method they use to prove this theorem is to show that $A$ contains large a subset $A'$ such that almost all products $ab$ with $a,b\in A'$ and $a<b$ are distinct. This method is related to what we will use to prove the lower bound in \Cref{mainthm} when $\delta<1/\log 4$.
	
	Theorem 1 of \cite{ford2018extremal} states that for any $D>7/2$ there is a set $A \subset [1, N] \cap \mathbb{Z}$ of size
	$$|A| \geq \frac{N}{(\log N)^{\mtc/2}(\log \log N)^D}$$
	for which $|A \cdot A| \sim |A|^2/2$ as $N \to \infty$. Of relevance to the present paper is the proof method of Ford. He starts out with a set 
	$$B:= \bigg\{m \in (N/2,N]: \mu^2(m) = 1, \omega(m) = k, \omega(m,t) \leq \frac{\log \log t}{\log 4} +2\;(3 \leq t \leq N)\bigg\}.$$
	After some intermediary steps, he is able to select a random subset $A$ of $B$ which satisfies the desired properties. In this way, the starting point of his proof is to work with a set of integers satisfying the condition $\omega(n,t) \leq \frac{\log \log t}{\log 4}$, which is similar to our condition $\delta \leq \frac{1}{\log 4}$. As such, Ford recognized that the set of extremal size one could construct which still satisfies a condition like $|A \cdot A| \sim |A|^2/2$ would be one resembling $\S_Q \cap [N/2,N]$ with $\density \leq \frac{1}{\log 4}$. Our \Cref{comb inter} supports this idea.
	
	\begin{remark}
		It is possible to relax the condition \eqref{definition of delta}. For example, we can fix some $\epsilon > 0$ and assume that the error term in \eqref{definition of delta} is $ \kappa x/(\log x)^{1+\epsilon}$. One could also replace the condition \eqref{definition of delta} with $\sum_{p \in Q \cap [1,x]}1/p = \delta \log\log x +M(Q) + O(1/(\log x))$ and add the assumption that we only work over families of $Q$ where $M(Q)$ is uniformly bounded. For simplicity's sake, we will instead work with condition \eqref{definition of delta}.
	\end{remark}
	
	\subsection*{Outline of the proof of \Cref{mainthm}} 
	We generally follow the methods of \cite{ford2008distribution}. Here, we give an outline of the proof, with the intent of highlighting that which is novel in our approach. 
	
	The proof is divided into \Cref{prop 1.4} and \Cref{prop 1.5}. The first proposition reduces the problem to the estimation of a Poisson-type sum, and the second proposition gives the appropriate estimates on said sum. 
	\begin{proposition}\label{prop 1.4}
		Let $c,\kappa, y_0$ be as in \Cref{mainthm}. Suppose $y_0 < y \leq \sqrt{x}$. One has, uniformly over all sets of primes $Q$ satisfying \eqref{definition of delta} and $\delta \ge c$, that
		\begin{equation}\label{main lemma}
			H_Q(x,y,2y) \asymp_{c,\kappa} \frac{x}{(\log x)^{1-\density}(\log y)^{1+\density}} \sum_{1 \leq k \leq v}\frac{\lambda^k}{k!}\cdot\frac{(v-k+1)}{v},
		\end{equation}
		where $\lambda = 2 \density \log \log y,$ and $v =  \lfloor\frac{1}{\log 2}\log \log y\rfloor.$
	\end{proposition}
	
	\begin{proposition}\label{prop 1.5}
		Let $c,y_0$ be as in $\Cref{mainthm}$, and let $\lambda$ and $v$ be as in \Cref{prop 1.4}. For $y > y_0$, we have uniformly for $\density \ge c$ that
		\begin{equation}
			\sum_{1 \leq k \leq v} \frac{\lambda^k}{k!}\cdot\frac{v-k+1}{v} \asymp_c \begin{cases}
				e^\lambda \max\Big\{\frac{1}{\log 4} - \density, \frac{1}{\sqrt{\log \log y}}\Big\} & \text{if } \density \leq \frac{1}{\log 4}, \\
				\noalign{\vskip 5pt}
				\frac{\lambda^v}{v!} \frac{1}{\max\Big\{\big(\density - \frac{1}{\log 4}\big)^2 \log \log y, 1\Big\}} & \text{if } \density > \frac{1}{\log 4}.
			\end{cases}
		\end{equation}
	\end{proposition}
	The proof of \Cref{prop 1.4} is outlined at the end of this section and proven in \Cref{sec lower bounds,sec upper bounds,sec lower bound proofs,sec upper bound proofs}. The proof of \Cref{prop 1.5} is given in \Cref{sec trans range}.
	
	\begin{proof}[Proof of \Cref{mainthm}]
		Write $\lambda = 2\density \log \log y$. By \Cref{prop 1.4} and \Cref{prop 1.5}, we have
		\begin{equation} \label{proof main thm eq 1}
			H_Q(x,y,2y) \asymp_{c,\kappa} \frac{x}{(\log x)^{1-\density}(\log y)^{1+\density}} \begin{cases}
				e^\lambda \max\big\{\frac{1}{\log 4} - \density, \frac{1}{\sqrt{\log \log y}}\big\} & \text{if } \density \leq \frac{1}{\log 4}, \\
				\noalign{\vskip 5pt}
				\frac{\lambda^v}{v!} \frac{1}{\max\big\{\big(\density - \frac{1}{\log 4}\big)^2 \log \log y, 1\big\}} & \text{if } \density > \frac{1}{\log 4}.
			\end{cases}
		\end{equation}
		For $\density \leq \frac{1}{\log 4}$, we substitute $e^\lambda = (\log y)^{2\density}$. The exponent of $\log y$ thus becomes $1+\density-2\density = 1-\density$, which directly matches $G(\density)$. 
		
		For $\density > \frac{1}{\log 4}$, an application of Stirling's formula yields
		$$ \frac{\lambda^v}{v!} \asymp \frac{1}{\sqrt{v}} \bigg(\frac{e\lambda}{v}\bigg)^v \asymp_c \frac{(\log y)^{1+\density - G(\density)}}{\sqrt{\log \log y}}.$$
		Substituting this into \eqref{proof main thm eq 1} directly recovers the $E(y;\density)$ term, completing the proof of the theorem.
	\end{proof}
	We finish this section with an outline of the proof of \Cref{prop 1.4}. The rigorous proof of this proposition is the content of \Cref{sec lower bounds,sec upper bounds,sec lower bound proofs,sec upper bound proofs}.
	\subsection*{Outline of the proof of \Cref{prop 1.4}}
	In this outline, all implicit constants may depend on $c,\kappa$, even if not indicated.
		\subsubsection*{Lower Bounds.} The first step is to relate $H_Q(x,y,2y)$ to a certain sum involving the function $L(a)$ which is defined in \eqref{def L}. The function $L(a)$ gives a quantitative hold on the propinquity (spacing)\footnote{See articles \cite{erdHos1979propinquity}, \cite{maier1984set} and \cite{raouj2011mesures} for an overview of the fascinating phenomenon of divisor propinquity.} of the divisors of $a$, and is frequently seen in the literature on divisors. We will show in \Cref{LB initial reduction} that
		\begin{equation}\label{outline eq 1}
			H_Q(x,y,2y) \gg {f}_\density(x,y)\sum_{a \in \mathscr{A}}\frac{L(a)}{a},
		\end{equation}
		for a certain set $\mathscr{A} \subseteq S_Q \cap \{a \leq y : \mu^2(a) = 1\}$, and a certain elementary function $f_\density$. For each $k \in \N$ we choose a certain subset $\mathscr{A}_k \subseteq \mathscr{A}$, where integers in $\mathscr{A}_k$ have exactly $k$ prime factors. The explicit construction of such sets $\mathscr{A}_k$ is discussed in \Cref{sec lower bounds}. The advantage of working with the carefully chosen sets $\mathscr{A}_k$ is that it will reduce the problem of bounding the sum of $L(a)/a$ from below into a strictly geometric problem. Indeed, \Cref{lem from L to Yk} implies that 
		$$\sum_{a \in \mathscr{A}_k}\frac{L(a)}{a} \gg (2 v \delta \log 2)^k \cdot 
		\Vol\big(\mathcal{Y}_k(v,C)\big),$$
		where $$v = \Big\lfloor\frac{ \log \log y}{\log 2}\Big\rfloor,$$ $C$ is a sufficiently large absolute constant, and $\mathcal{Y}_k(v,C)$ is a geometric sub-region of the $k$-dimensional unit hypercube which is defined in \Cref{lem from L to Yk}. 
		
		The final step is computing a lower bound on the volume of this region. In \Cref{lem LB Yk}, we show that
		$$ \Vol\big(\mathcal{Y}_k(v,C)\big) \gg \frac{v-k+1}{v \cdot k!}\quad (1 \leq k \leq v). $$ 
		Our proof of this lower bound introduces a new perspective by exploiting a phenomenon well-known in the probability literature: that random walks conditioned to satisfy a weak barrier constraint will, with high probability, respect a much stronger barrier condition (see e.g. \cite{arguin2017maxima} for an example of this phenomenon being exploited in the context of number theory, or \cite{bramson1978maximal} for a classical and purely probabilistic example). While this principle was implicitly present in the arguments of Ford \cite{ford2008distribution} and Koukoulopoulos \cite{koukoulopoulos2010generalized}, we make this geometric intuition explicit. By defining the region $\mathcal{Y}_k(v,C)$ to obey \textit{strong barrier conditions}, we illuminate the geometric mechanism underlying the lower bound. See the remark following \Cref{lem from L to Yk} for more details on our conceptual improvement.

		With a lower bound for $\Vol(\mathcal{Y}_k(v,C))$ in hand, we simply sum over all $k \leq v$ and arrive at a lower bound for $H_Q(x,y,2y)$. To summarize the lower bound proof schematically, we show:
		\begin{align*}
			\frac{H_Q(x,y,2y)}{f_\density(x,y)} &\gg \sum_{\substack{a \in \mathscr{A}_k \\ k \geq 1}}\frac{L(a)}{a} && \text{Lemmas 3.1-3.4}\\ &\gg \sum_{k \leq v} 2^{-C} (2 v \delta \log 2)^k \Vol\big(\mathcal{Y}_k(v,C)\big) && \text{Lemma 3.5}\\ &\gg \sum_{k \leq v}\frac{\lambda^k}{k!}\cdot\frac{v-k+1}{v},&& \text{Lemma 3.6}
		\end{align*}
		where $\lambda = 2 \density \log \log y$ and $v = \frac{1}{\log 2}\log \log y$.
		
		\subsubsection*{Upper Bounds.}The upper bound proof begins by applying properties of the function \(L\) to relate \(H_Q(x,y,2y)\) to a sum involving \(L(a)\). \Cref{lemma L properties,lem UB 1,lem UB 3} imply that 
		$$H_Q(x,y,2y) \ll {f}_\density(x,y)\sum_{a \in \mathscr{B}}\frac{L(a)}{a},$$
		for a certain set $\mathscr{B} \subseteq S_Q$, and the very same $f_\delta$ as in \eqref{outline eq 1}. We then partition $\mathscr{B} = \cup_k \mathscr{B}_k,$ where $\mathscr{B}_k = \{a \in \mathscr{B}: \omega(a) = k\}$. \Cref{lem UB 4} provides a bound of the form
		$$\sum_{a \in \mathscr{B}_k}\frac{L(a)}{a} \ll (2 \delta \log \log y)^k U_k(v),$$
		where $U_k(v)$ is a multivariate volume integral defined in \eqref{def Uk}. Uncoincidentally, the same integral $U_k(v)$ appears in the past works of Ford \cite{ford2006integers} and Koukoulopoulos \cite{koukoulopoulos2010generalized}. An appropriate upper bound on $U_k(v)$ for each $k$ follows immediately from their work, as is noted in \Cref{lem UB 5}. It will follow that the dominating contribution to our upper bound comes from $k \leq v$, and that for such $k$ we have
		$$U_k(v) \ll \frac{v-k+1}{v \cdot k!}.$$
		Summing over all such $k$ will yield an upper bound which matches the lower bound. To summarize the upper bound proof schematically, we show that
		\begin{align*}
			\frac{H_Q(x,y,2y)}{f_\density (x,y)} &\ll \sum_{\substack{a \in \mathscr{B}_k \\ k \geq 1}} \frac{L(a)}{a} &&\text{Lemmas 4.1-4.3}\\ &\ll \sum_{k \geq 1} (2 \delta \log \log y)^k U_k(v) && \text{Lemma 4.4}\\ &\ll \sum_{k \leq v}\frac{\lambda^k}{k!}\cdot\frac{v-k+1}{v}. && \text{Lemma 4.5}
		\end{align*}

	\subsection*{Notation} Throughout the paper we use standard Vinogradov and Oh asymptotic notation: $\ll,\gg,\sim,\asymp,O(\cdot),o(\cdot)$. The variables $p,q$ always denote prime numbers. The function $\log_kx$ stands for the $k$-fold iterated logarithm of $x$ (for example, $\log_2 x = \log \log x$). The function $\omega(n)$ denotes the number of distinct prime factors of $n$. We give at the very end of the paper (\Cref{table:notation}) an index of some of the most important symbols used throughout the text.
	
	\subsection*{Acknowledgments}
	The author is indebted to Dimitris Koukoulopoulos for his guidance and mentorship throughout the project, and for the initial idea. We thank Kevin Ford for useful feedback on an earlier version of this manuscript. We thank Paul Bourgade for pointing us towards some relevant literature on ballot theorems. Countless thanks are due to Cihan Sabuncu for myriad helpful conversations. The author also benefited from discussions with Sun Kai Leung and Tony Haddad. Lastly, the author is eternally grateful to Stephanie Tanasia for her endless support.
	
	\section{Useful Lemmas}
	\begin{lemma}\label{defining Q} Let $\kappa \ge 2$, and let $Q$ be a set of primes as in \eqref{definition of delta}. That is: 
		$$\Big|\#\big(Q\cap[1,x]\big) - \frac{ \delta x}{\log x}\Big|\le \frac{\kappa x}{(\log x)^2},$$ for all $x\ge2$. Then
		\begin{equation}
			\sum_{\substack{p \in Q\\ p \leq x}} \frac{1}{p} = \density \log \log x + C(Q) +O_\kappa\bigg(\frac{1}{\log x}\bigg),
		\end{equation}
		where $C(Q)$ is a constant satisfying $C(Q) \ll \kappa$.
	\end{lemma}
	\begin{proof}
		This is a standard application of partial summation to \eqref{definition of delta}. Nonetheless, we include a full proof to emphasize the fact that the constants $C(Q)$ are bounded by $\kappa$. Let us write $\pi_Q(x) := \sum_{\substack{p \in  Q \cap[1,x]}} 1$. By assumption, we have
		\begin{equation}\label{pi Q error}
			\pi_Q(x) = \frac{\density x}{\log x} + E_Q(x), \quad (x \geq 2)
		\end{equation}
		where $|E_Q(x)| \leq \kappa x/(\log x)^2$. We apply partial summation:
		$$\sum_{\substack{p \leq x \\ p \in Q}} \frac{1}{p} = \delta \int_2^x \frac{dt}{t \log t}+ \frac{E_Q(t)}{t}\Big|_{t=2^-}^x + \int_2^x \frac{E_Q(t)}{t^2} dt.$$
		By \eqref{pi Q error}, it follows that
		\begin{align*}
			\sum_{\substack{p \leq x \\ p \in Q}} \frac{1}{p} &=  \density \log \log x+C(Q) + O_\kappa\left(\frac{1}{\log x}\right)
		\end{align*}
		where $$C(Q) :=\int_2^\infty \frac{E_Q(t)}{t^2}dt -\frac{E_Q(2^-)}{2}-\density \log \log 2 .$$
		By substituting $|E_Q(t)| \leq \kappa t/(\log t)^2$ into the above, we find that ${|C(Q)| \ll \kappa}$.\end{proof}

	\begin{lemma}[\cite{koukoulopoulos2019distribution}, Exercise 14.5]\label{dim Brun}
		Let $f$ be a multiplicative function with $0 \leq f \leq \tau_k$ for some $k \in \N$. Further suppose that there exists $c>0$ such that
		$$\sum_{p \leq y}f(p) \geq \frac{c y}{\log y}\quad (\sqrt{x} \leq y \leq x).$$
		One then has
		$$\sum_{n \leq x}f(n) \asymp_{c} x \cdot \exp\bigg\{\sum_{p \leq x} \frac{f(p) - 1}{p}\bigg\}.$$
	\end{lemma}
	\begin{proof}
		We will first show that \begin{equation*}
			\sum_{n \leq x}\frac{f(n)}{n} \asymp_k \exp\bigg(\sum_{p \leq x} \frac{f(p)}{p} \bigg).
		\end{equation*}
		The upper bound in the above equation follows from an Euler product expansion:
		\begin{equation}\label{ex 14.5 eq 6}
			\sum_{n \leq x}\frac{f(n)}{n} \le \exp\bigg(\sum_{p \leq x} \log \Big(1+\frac{f(p)}{p}+\frac{f(p)^2}{p^2}+\cdots \Big)\bigg) \ll_k \exp\bigg(\sum_{p \leq x} \frac{f(p)}{p} \bigg),
		\end{equation}
		since $\log(1+f(p)/p+\cdots) = f(p)/p +O_k(1/p^2)$. 
		
		The lower bound requires a bit more work. Define $g$ to be the arithmetic function for which $f * g = \tau_k$, where $*$ denotes the Dirichlet convolution.  Note that $g$ is multiplicative, since the set of multiplicative functions is closed under Dirichlet convolution. By definition, we have $g(p) = \tau_k(p) - f(p) \geq 0$. Since $g(p) \ge 0$ for all primes and $g$ is multiplicative, then $\mu^2(n)g(n) \geq 0$ for all positive integers $n$. Since we also have $\mu^2(n) f(n) \ge 0$, then we can write
		
		\begin{align*}
			\bigg(\sum_{n \leq x} \frac{\mu^2(n)f(n)}{n}\bigg)\bigg(\sum_{m \leq x} \frac{\mu^2(m)g(m)}{m}\bigg) &\geq \sum_{\substack{n,m \le x \\ nm \le x \\ (n,m) = 1}} \frac{\mu^2(n)\mu^2(m)f(n)f(m)}{nm} \\ &= \sum_{\substack{n,m \le x \\ nm \le x \\ (n,m) = 1}} \frac{\mu^2(nm)f(n)f(m)}{nm},
		\end{align*}
		since $\mu^2$ is multiplicative. By writing $a = nm$ in the above sum, we see that
		\begin{equation*}
			\sum_{\substack{n,m \le x \\ nm \le x \\ (n,m) = 1}} \frac{\mu^2(nm)f(n)f(m)}{nm} = \sum_{a \leq x} \frac{\mu^2(a)}{a}\sum_{nm=a} f(n)f(m) =  \sum_{a \leq x} \frac{\mu^2(a)}{a} (f*g)(a).
		\end{equation*}
		Using the fact that $f*g = \tau_k$, we conclude that
		\begin{equation}\label{ex 14.5 eq 1}
			\bigg(\sum_{n \leq x} \frac{\mu^2(n)f(n)}{n}\bigg)\bigg(\sum_{m \leq x} \frac{\mu^2(m)g(m)}{m}\bigg) \geq \sum_{a \leq x} \frac{\mu^2(a) \tau_k(a)}{a}.
		\end{equation}
		
		By using the non-negativity of $f$, together with \eqref{ex 14.5 eq 1}, we see that
		\begin{equation}\label{ex 14.5 eq 2}
			\sum_{n \leq x}\frac{f(n)}{n} \geq \bigg(\sum_{a \leq x} \frac{\mu^2(a) \tau_k(a)}{a}\bigg)\bigg(\sum_{m \leq x} \frac{\mu^2(m)g(m)}{m}\bigg)^{-1}.
		\end{equation}
		Since $g(p) = \tau_k(p) -f(p) = k-f(p),$ we can can use the inequality $1+t \leq e^t$ to write
		\begin{equation}\label{ex 14.5 eq 3}
			\sum_{m \leq x}\frac{\mu^2(m)g(m)}{m} \le \exp\bigg(\sum_{p \leq x}\log\Big(1+\frac{g(p)}{p}\Big)\bigg) \le  \exp\bigg(\sum_{p \leq x}\frac{k-f(p)}{p}\bigg).
		\end{equation}
		On the other hand, by standard analytic estimates, we have
		\begin{equation}\label{ex 14.5 eq 4}
			\sum_{a \leq x}\frac{\mu^2(m)\tau_k(a)}{a}  \gg_k (\log x)^k \asymp_k \exp\bigg(\sum_{p \leq x}\frac{k}{p}\bigg).
		\end{equation}
		By substituting \eqref{ex 14.5 eq 3} and \eqref{ex 14.5 eq 4} into \eqref{ex 14.5 eq 2}, we find that
		\begin{equation}\label{ex 14.5 eq 5}
			\sum_{n \leq x}\frac{f(n)}{n} \gg_k \exp\bigg(\sum_{p \leq x}\frac{f(p)}{p}\bigg).
		\end{equation}
		The equations \eqref{ex 14.5 eq 5} and \eqref{ex 14.5 eq 6} together show that
		\begin{equation}\label{ex 14.5 eq 7}
			\sum_{n \leq x}\frac{f(n)}{n} \asymp_k \exp\bigg(\sum_{p \leq x} \frac{f(p)}{p} \bigg).
		\end{equation}
		
		To finish the proof, we Note that all products $mp$ with $m\le x^{1/3}$ and $p>x^{1/3}$ are distinct. Hence, we can bound the sum from below by isolating these terms:
		\begin{equation}\label{ex 14.5 eq 8}
			\sum_{n \leq x}f(n) \geq \sum_{m \leq x^{1/3}} f(m) \sum_{x^{1/3} < p \leq x/m} f(p).
		\end{equation}
		Since $m \leq x^{1/3}$, the upper limit of the inner sum satisfies $x/m \geq x^{2/3} > \sqrt{x}$. This places $x/m$ inside the domain where our hypothesis $\sum_{p \leq y} f(p) \geq c y/\log y$ holds. We can therefore estimate the inner sum by writing:
		\begin{equation*}
			\sum_{x^{1/3} < p \leq x/m} f(p) = \sum_{p \leq x/m} f(p) - \sum_{p \leq x^{1/3}} f(p) \geq c \frac{x/m}{\log(x/m)} - O_k\bigg(\frac{x^{1/3}}{\log x}\bigg).
		\end{equation*}
		Because $\log(x/m) \asymp \log x$ and the error term is negligible relative to the main term when $x/m \geq x^{2/3}$, the inner sum is $\gg_{c,k} \frac{x}{m \log x}$ for sufficiently large $x$. Substituting this back into \eqref{ex 14.5 eq 8} yields
		\begin{equation*}
			\sum_{n \leq x} f(n) \gg_{c,k} \frac{x}{\log x} \sum_{m \leq x^{1/3}} \frac{f(m)}{m}.
		\end{equation*}
		By our result in \eqref{ex 14.5 eq 7}, the sum over $m$ is $$\asymp_k \exp\Big(\sum_{p \leq x^{1/3}} \frac{f(p)}{p}\Big).$$ We can extend this prime sum up to $x$ at the cost of a constant factor depending on $k$, since $$\sum_{x^{1/3} < p \leq x} \frac{f(p)}{p} \leq k \sum_{x^{1/3} < p \leq x} \frac{1}{p} \ll_k 1.$$ Using Mertens' theorem, $1/\log x \asymp \exp\big(-\sum_{p \leq x} \frac{1}{p}\big)$, we arrive at the lower bound:
		\begin{equation}\label{ex 14.5 eq 9}
			\sum_{n \leq x} f(n) \gg_{c,k} x \cdot \exp\bigg(\sum_{p \leq x} \frac{f(p)-1}{p}\bigg).
		\end{equation}
		
		For the upper bound, we use estimates on sums of multiplicative functions which appear in the literature. By \cite[{Theorem 14.2}]{koukoulopoulos2019distribution}, we have $$\sum_{n \leq x} f(n) \ll_k x \exp \bigg(\sum_{p \leq x} \frac{f(p) - 1}{p}\bigg).$$
		This completes the proof of the lemma.
	\end{proof}
	\begin{lemma}\label{count s P-}
		Let $c,\kappa,y_0$ be as in \Cref{mainthm}. Let $x > y_0,$ and suppose $z \in [3/2,x/2].$  Uniformly over all sets of primes $Q$ satisfying \eqref{definition of delta} with $\density \ge c$, we have 
		$$\sum_{\substack{n \leq x\\ n \in \S_Q \\ P^-(n) > z}}1 \asymp_{c,\kappa} \frac{x}{(\log x)^{1-\density}(\log z)^\density}.$$
	\end{lemma}
	\begin{proof}
		Let $f(n)$ be the indicator function of the set $\S \cap \{n : P^{-}(n) > z\}$.
		
		\par First let us consider the case $z \leq \sqrt{x}/2$. For $\sqrt{x} < y \leq x$, we have $y > 2z$. Thus, for $y$ sufficiently large in terms of $c,\kappa$, we ensure that
		$$\sum_{p \leq y} f(p) \geq c \cdot \frac{ y}{\log y}.$$
		Consequently, by \Cref{dim Brun} we have for all sufficiently large $x$ that
		\begin{align}
			\sum_{n \leq x}f(n) \asymp_{c,\kappa} x \cdot \exp\bigg\{-\sum_{\substack{p \leq x \\ p \not\in \S}} \frac{1}{p} -\sum_{\substack{p \leq z \\ p \in \S}} \frac{1}{p} \bigg\}. \label{proof lem sieve eq 1}
		\end{align}
		By \Cref{defining Q},
		\begin{align*}
			\sum_{\substack{p \leq z \\ p \in \S}} \frac{1}{p} &= \density \log_2 z +C(Q) +o_\kappa(1),
		\end{align*}
		and
		\begin{align*}
			\sum_{\substack{p \leq x \\ p \not\in \S}} \frac{1}{p} &= (1-\density) \log_2 x +C'(Q) +o_\kappa(1).
		\end{align*}
		By Mertens' theorem, $C'(Q) = M'-C(Q),$ where $M'$ is Mertens' constant. Hence, by \eqref{proof lem sieve eq 1} we have
		\begin{align*}
			\sum_{n \leq x}f(n)&\asymp_c x \cdot \exp\big\{-(1-\density)\log_2 x -\density \log_2 z -M' +o_{\kappa}(1)\big\}\\
			&\asymp_{c,\kappa} \frac{x}{ (\log x)^{1-\density} (\log z)^\density}.
		\end{align*}
		
		\par Lastly, let us consider the case $z \in [\sqrt{x}/2,x/2]$. In this range, we have $\log z \asymp \log x$. By standard sieve upper bounds, we have
		\begin{align*}
			\sum_{n \leq x}f(n) &\ll_{c,\kappa} \frac{x}{\log z} \\
			&\asymp_{c,\kappa} \frac{x}{ (\log x)^{1-\density} (\log z)^\density}.
		\end{align*}
		For the lower bound, we restrict our count to primes. Since $f(p) = 1$ for all primes $p \in \S$ such that $p > z$, we have
		\begin{align*}
			\sum_{n \leq x}f(n) \geq \sum_{\substack{z < p \leq x \\ p \in \S}} 1 &\asymp_{c,\kappa} \frac{x}{\log x} \\
			&\asymp_{c,\kappa} \frac{x}{ (\log x)^{1-\density} (\log z)^\density}. \qedhere
		\end{align*}
	\end{proof}
	The next lemma appears in a more general form in \cite{koukoulopoulos2010localized}. We state it in a form which is more specific to our context. Note that the variable we call $2y$ is written as $x$ in the mentioned reference.
	\begin{lemma}[{\cite[Lemma 2.2 (b)]{koukoulopoulos2010localized}}] \label{dimitris lemma 2.2}
		Let $f : \N \to [0,+\infty)$ be an arithmetic function. Assume there exists a constant $C_f$ depending only on $f$ such that $f(ap) \leq C_f f(a)$ for all $a \in \N$ and all primes $p$ with $(a,p) = 1$. Let $3/2 \leq y, h \geq 0, \epsilon > 0$. Then
		$$\sum_{\substack{a: P^+(a) \leq 2y \\ \mu^2(a) = 1}}\frac{f(a)}{a \log^h(P^+(a)+(2y)^\epsilon/a)} \ll_{C_f,h,\epsilon} \frac{1}{(\log y)^h}\sum_{\substack{a: P^+(a) \leq 2y \\ \mu^2(a) = 1}}\frac{f(a)}{a}.$$
	\end{lemma}

	\section{The Lower Bound in Proposition 1.4}\label{sec lower bounds}
	In this section, we prove the lower bounds implicit in \Cref{mainthm}. Define
	\begin{equation}\label{def L}
		\mathscr{L}(a) = \bigcup_{d|a}(\log d - \log 2, \log d]; \quad L(a) = \text{meas} \big(\mathscr{L}(a)\big),
	\end{equation}
	where `meas' denotes the Lebesgue measure on $\R$. As a first step towards the lower bound, we relate $H(x,y,2y)$ to a sum involving $L(a)$. The following lemma accomplishes this by generalizing Lemma 4.1 of \cite{ford2008distribution}.
	
	\begin{lemma}\label{LB initial reduction}
	Let $c,\kappa,y_0$ be as in \Cref{mainthm}. Suppose that $y_0 < y \leq \sqrt{x}$. Then uniformly over all sets of primes $Q$ satisfying \eqref{definition of delta} with $\density \ge c$, we have  
		$$H_Q(x,y,2y) \gg_{c,\kappa} \frac{x}{(\log x)^{1-\density}(\log y)^{1+\density}} \sum_{\substack{a \leq y^{1/8} \\ a \in \S \\ \mu^2(a) = 1}} \frac{L(a)}{a}.$$
	\end{lemma}
	\Cref{LB initial reduction} will be proven in \Cref{sec lower bound proofs}. We bound the sum over $L(a)/a$ from below using the Cauchy-Schwarz inequality. This is the content of the next lemma, which is the same as Lemma 2.2 in \cite{ford2006integers}.
	\begin{lemma}\label{cauchy schwarz L}
		For any given $\mathscr{A} \subseteq \N,$ one has 
		\begin{equation}\label{CS 1}
			\sum_{a \in \mathscr{A}} \frac{L(a)}{a} \gg  \frac{\left(\sum_{a \in \mathscr{A}} \frac{\tau(a)}{a}\right)^2}{\left( \sum_{a \in \mathscr{A}} \frac{W(a)}{a}\right)},
		\end{equation}
		where \begin{equation}
			W(a) := \#\{(d,d'):d|a,d'|a, |\log(d/d')| \leq \log 2\}.
		\end{equation}
	\end{lemma}
	
	\Cref{LB initial reduction} and \Cref{cauchy schwarz L} reduce the problem of finding a lower bound for $H_Q(x,y,2y)$ to that of finding a set $\mathscr{A}$ of integers for which $\sum_{a \in \mathscr{A}} \tau(a)/a$ and $\sum_{a \in \mathscr{A}} W(a)/a$ can be sharply bounded from below and from above respectively. We now define such a set $\mathscr{A}$. We remark that our specific definition allows much of Ford's analysis on the $\delta = 1$ case in \cite{ford2008distribution} to directly carry over. Define $\lambda_j$ as follows: Let $\lambda_0 = 1.9$, and let $\lambda_j$ be the largest prime number in $Q$ for which
	\begin{equation}\label{define lambda}
		\sum_{\substack{\lambda_{j-1}< p \leq \lambda_j \\ p \in Q}} \frac{1}{p} \leq \density \log 2.
	\end{equation}
	In addition, define $D_j := (\lambda_{j-1},\lambda_j]$. Given a vector $\bm{b} = (b_1,\dots, b_h)$, define $\mathscr{A}(\bm{b})$ to be the set of squarefree integers which have $b_i$ prime factors in the interval $D_i$. When $\bm{b}$ is fixed, we write $k := b_1 + \cdots +b_h$ (i.e.~all numbers in $\mathscr{A}(\bm{b})$ have the same total number of prime factors, which we call $k$). 
	\begin{remark}
		The above construction implies that $\mathscr{A}(\bm{b}) \subseteq \S_Q$ for each vector $\bm{b}$. Furthermore, if $\bm{b}_1 \ne \bm{b}_2$ then $\mathscr{A}(\bm{b}_1) \cap \mathscr{A}(\bm{b}_2) = \emptyset$. In this way, $\cup_{\bm{b}} \mathscr{A}(\bm{b})$ forms a partition of the squarefree integers in $\S_Q$ if the union is taken over all vectors $\bm{b} \in \bigcup_{h \geq 1}(\mathbb{Z}_{\geq 0})^h$.
	\end{remark}
	Before proceeding, we prove a lemma on the size of $\lambda_j$.
	\begin{lemma}\label{lem: size of lambda}
		Let $c,\kappa$ be as in \Cref{mainthm}. Then there exists a constant $R =R(c,\kappa)$ depending only on $c,\kappa$ such that uniformly over $j \geq 1$, one has
		$$2^{j-R} \leq \log(\lambda_j) \leq 2^{j+R}.$$
	\end{lemma}
	\begin{proof}
		By the definition of $\lambda_j$, it is the largest prime in $Q$ such that the sum of reciprocals is $\leq \density \log 2$. Therefore, adding the next available prime $q \in Q$ to the sum would strictly exceed $\density \log 2$. Since $1/q \ll 1/\lambda_j \leq 1/\log \lambda_{j-1}$, we have
		\begin{equation}\label{eq: size of lambda 1}
			\sum_{\substack{\lambda_{j-1}< p \leq \lambda_j \\ p \in Q}} \frac{1}{p} = \density \log 2 + O\Big(\frac{1}{\log \lambda_{j-1}}\Big).
		\end{equation}
		By \Cref{defining Q}, we also have the asymptotic expansion
		\begin{equation}\label{eq: size of lambda 2}
			\sum_{\substack{\lambda_{j-1}< p \leq \lambda_j \\ p \in Q}} \frac{1}{p} = \delta(\log \log \lambda_j - \log \log \lambda_{j-1}) +O_{\kappa}\Big(\frac{1}{\log \lambda_{j-1}}\Big).
		\end{equation}
		Since $\density \geq c$, we can combine \eqref{eq: size of lambda 1} and \eqref{eq: size of lambda 2} and divide by $\delta$ to deduce that there exists a constant $r=r(c,\kappa)>0$ for which 
		\begin{equation*}
			\big|\log \log \lambda_j - \log \log \lambda_{j-1} - \log 2\big| \leq \frac{r}{\log \lambda_{j-1}}.
		\end{equation*}
		Exponentiating both sides yields
		$$ 2 (\log \lambda_{j-1}) e^{-\frac{r}{\log{\lambda_{j-1}}}} \leq \log \lambda_j \leq 2 (\log \lambda_{j-1}) e^{\frac{r}{\log{\lambda_{j-1}}}}.$$
		Applying this recursively down to $\lambda_0 = 1.9$, we see that 
		$$ 2^j (\log \lambda_0) \exp\bigg({-r\sum_{i = 0}^{j-1} \frac{1}{\log{\lambda_{i}}}}\bigg) \leq \log \lambda_j \leq 2^j (\log \lambda_0) \exp\bigg({r\sum_{i = 0}^{j-1} \frac{1}{\log{\lambda_{i}}}}\bigg).$$
		Because the recurrence implies $\log \log \lambda_i - \log \log \lambda_{i-1} \sim \log 2$, we have the coarse lower bound $\log \lambda_i \gg (1.5)^i$. Thus, the series $\sum_{i=0}^\infty \frac{1}{\log \lambda_i}$ converges rapidly. The exponential factors, along with the constant $\log \lambda_0$, can therefore be strictly bounded above and below by constants depending only on $c$ and $\kappa$. Absorbing these into a single parameter $2^{\pm R}$ completes the proof.
	\end{proof}
	
	 Upper and lower bounds can now be calculated for the $\tau$ and $W$ sums of \eqref{CS 1} when $\mathscr{A} = \mathscr{A}(\bm{b})$ for an appropriate type of $\bm{b}$. This is the content of the next two lemmas, which will be proved in \Cref{sec lower bound proofs}. 
	
	\begin{lemma}\label{LB tau} Let $c,\kappa$ be as in \Cref{mainthm}. Let $M = M(c,\kappa)$ be a sufficiently large constant depending only on $c,\kappa$. Suppose that $b_j = 0$ for all $j < M$, and $b_j \leq M j$ for all $j$. One has
		\begin{equation}
			\sum_{a \in \mathscr{A}(\bm{b})} \frac{\tau(a)}{a} \gg_{c,\kappa} \frac{(2 \density \log 2)^k}{b_M! \cdots b_h!}.
		\end{equation}
	\end{lemma}
	
	\begin{lemma}\label{LB W} Let $c,\kappa$ be as in \Cref{mainthm}, and let $M=M(c,\kappa)$ be as in \Cref{LB tau}. For $b_j$ satisfying the conditions of \Cref{LB tau}, one has
		\begin{equation}
			\sum_{a \in \mathscr{A}(\bm{b})} \frac{W(a)}{a} \ll_{c,\kappa} \frac{(2 \density \log 2)^k}{b_M! \cdots b_h!}\left[ 1 + \sum_{j = M}^h 2^{-j +b_M + \cdots +b_j} \right].
		\end{equation}
	\end{lemma}
	One can think of the condition $b_j \leq Mj$ as being a mild restriction which ensures that $a \leq y^{1/8}$ for each $a \in \mathscr{A}(\bm{b})$. Lemmas \ref{cauchy schwarz L}, \ref{LB tau}, and \ref{LB W} lead to the estimate
	\begin{equation}\label{L to factorials}
		\sum_{\substack{a \in \mathscr{A}(\bm{b})}} \frac{L(a)}{a} \gg_{c,\kappa}  \frac{(2 \density \log 2)^k}{b_M! \cdots b_h!} \bigg[ 1 + \sum_{j =M}^h 2^{-j +b_M + \cdots +b_j} \bigg]^{-1}.
	\end{equation}
	We wish to further restrict our choice of $\bm{b}$ in such a way as to force the complicated $j$ sum of \eqref{L to factorials} to be bounded. The factor $(b_M! \cdots b_h!)^{-1}$ is the volume of some geometric region in $\mathbb{R}^k$, and so by restricting our choice of $\bm{b}$, we eliminate the $j$ sum from our lower bound at the cost of making the evaluation of this volume more complicated. This is the content of the next lemma, which will be proved in \Cref{sec lower bound proofs}.
	
	\begin{lemma}\label{lem from L to Yk}Let $c,\kappa$ be as in \Cref{mainthm}, and let $M=M(c,\kappa)$ be as in \Cref{LB tau}. Define $$\Tilde{v}:=v-2M=\Big\lfloor\frac{\log_2 y}{\log 2}\Big\rfloor - 2M,$$ and suppose $M+1 \leq k \leq v.$ Fix $C > 2$. Then for some set $\mathscr{A}_k \subseteq \S_Q$ of squarefree integers $a \leq y^{1/8}$ satisfying $\omega(a) = k$, we have
		$$\sum_{\substack{a \in \mathscr{A}_k}} \frac{L(a)}{a} \gg_{c,\kappa} (2v \density \log 2)^k \frac{\Vol(\mathcal{Y}_k(\Tilde{v},C))}{2^C}, $$
		where $\mathcal{Y}_k(\Tilde{v},C)$ is defined to be the set of $\bm{\xi} \in \mathbb{R}^k$ such that
		\begin{enumerate}[label=(\roman*)]
			\item $0 \leq \xi_1 \leq \cdots \leq \xi_k \leq 1$;
			\item $\xi_{M+i^2} > \frac{i}{\Tilde{v}}$ and $\xi_{k+1-(M+i^2)} < 1-\frac{i}{\Tilde{v}}$ for $1 \leq i \leq \sqrt{k-M}$;
			\item $\Tilde{v} \xi_i \geq \max\{i-1, i+ (\min\{i,k-i\})^{1/7}-C\}.$
		\end{enumerate}    
	\end{lemma}
	
\begin{remark}
	The condition (iii) represents a conceptual refinement of the methods found in the works of Ford \cite{ford2008distribution} and Koukoulopoulos \cite{koukoulopoulos2010generalized}. In those works, the analysis relies on a region $Y_k(\tilde{v},R)$, defined analogously to our $\mathcal{Y}_k(\tilde{v},C)$ but with condition (iii) replaced by the algebraic constraint $\sum_{i=1}^k 2^{i-v\xi_i} \le R$ for a large constant $R$. While effective for bounding $W(a)$, this implicit condition masks the underlying geometric behavior of the order statistics.
	
	In contrast, we impose a \textit{strong barrier condition}, requiring the normalized statistics $v\xi_i$ to repel from the mean $i$ by a margin of roughly $(\min\{i, k-i\})^{1/7}$. This explicit geometric constraint implies the necessary algebraic bound automatically. Indeed, the barrier $v\xi_i \geq i + (\min\{i, k-i\})^{1/7} - C$ forces the exponents to be sufficiently negative, so that
	$$ \sum_{i=1}^k 2^{i-v\xi_i} \leq 2^C \sum_{i=1}^k 2^{-(\min\{i, k-i\})^{1/7}} = O_C(1). $$
	By defining the region $\mathcal{Y}_k(\tilde{v},C)$ via these strong barriers, we flesh out the geometric intuition implicit in previous works. This allows us to leverage the powerful Smirnov statistic bounds established in Section 6, which are naturally suited to such barrier conditions.
\end{remark}
	The following lemma provides a lower bound for the volume of $\mathcal{Y}_k(\Tilde{v},R)$ which is sharp for the values of $k$ we are interested in.
	\begin{lemma}\label{lem LB Yk} Let $c,\kappa$ be as in \Cref{mainthm}. Suppose $k \leq \Tilde{v}$. When $C$ is a sufficiently large absolute constant, one has
		$$\Vol(\mathcal{Y}_k(\Tilde{v},C)) \gg \frac{v-k+1}{v \cdot k!}.$$
	\end{lemma}
	By combining \Cref{LB initial reduction}, \Cref{lem from L to Yk}, and \Cref{lem LB Yk}, and summing over the valid range $M+1 \leq k \leq \Tilde{v}$, we obtain the lower bound
	\begin{equation*}
		H_Q(x,y,2y) \gg_{M,c,\kappa} \frac{x}{(\log x)^{1-\density}(\log y)^{1+\density}}    \sum_{M+1 \leq k \leq \Tilde{v}} \frac{\lambda^k}{k!}\cdot\frac{(v-k+1)}{v},
	\end{equation*}
	where $\lambda := 2\density \log \log y$. To extend the sum to the full range $1 \leq k \leq v$, we observe that the missing terms at the boundaries are bounded by a constant factor relative to the active sum. Let $a_k$ denote the $k$-th term of the sum. 
	
	For the lower range $1 \leq k \leq M$, we trivially have $a_k \leq \lambda^M$. Because $M$ is fixed, $v \asymp \log \log y$, and $\lambda \to \infty$ as $y \to \infty$, the first active term satisfies $a_{M+1} \asymp_M \lambda^{M+1}$. Thus, for sufficiently large $y$, the sum of the first $M$ terms is $O_M(\lambda^M) = o(a_{M+1})$, making their contribution negligible. 
	
	For the upper range $\Tilde{v} < k \leq v$, there are exactly $2M$ missing terms. We bound each of these by comparing it directly to the final active term $a_{\Tilde{v}}$. For $k = \Tilde{v}+j$ with $1 \leq j \leq 2M$, we have
	$$ \frac{a_{\Tilde{v}+j}}{a_{\Tilde{v}}} = \frac{\lambda^j}{(\Tilde{v}+1)\cdots(\Tilde{v}+j)} \cdot \frac{v-\Tilde{v}-j+1}{v-\Tilde{v}+1}. $$
	Since $j \leq 2M$, the second fraction is strictly $\leq 1$. For the first fraction, because $\lambda \asymp v \asymp \Tilde{v}$, each factor $\frac{\lambda}{\Tilde{v}+i} \asymp 1$. Therefore, the entire ratio is bounded by an absolute constant depending only on $M$. Since there are exactly $2M$ such terms, their total sum is $\ll_M a_{\Tilde{v}}$.
	
	Extending the sum therefore increases its value by at most a factor depending on $M$. Since $M$ was chosen depending only on $c$ and $\kappa$, we absorb this dependence into our asymptotic notation to conclude that
	\begin{equation}\label{final LB}
		H_Q(x,y,2y) \gg_{c,\kappa} \frac{x}{(\log x)^{1-\density}(\log y)^{1+\density}}    \sum_{1 \leq k \leq v} \frac{\lambda^k}{k!}\cdot\frac{(v-k+1)}{v},
	\end{equation}
	where $\lambda := 2\density \log \log y$. This completes the proof of the lower bound in \Cref{prop 1.4}.
	
	\section{The Upper Bound in Proposition 1.4}\label{sec upper bounds}
	We begin with the following lemma, which contains several useful properties of the function $L(a)$.
	\begin{lemma}[\cite{ford2008distribution}, Lemma 3.1]\label{lemma L properties}
		One has
		\begin{enumerate}[label = (\roman*)]
			\item $L(a) \leq \min\{\log 2 \cdot \tau(a),\log 2 + \log a\}$
			\item If $(a,b) = 1$, then $L(ab) \leq \tau(b)L(a)$ 
			\item if $p_1<\cdots < p_k$ are primes, then 
			$$L(p_1\cdots p_k) \leq \min_{0 \leq j \leq k}2^{k-j}(\log(p_1\cdots p_k)+\log 2).$$
		\end{enumerate}
	\end{lemma}
	We define $$\mathscr{P}(z):= \big\{a \in \N: P^+(a) \leq z, \mu^2(a) = 1\big\}.$$

	\begin{lemma}\label{lem UB 1} Let $c,\kappa,y_{0}$ be as in \Cref{mainthm}. Suppose $y_0 < y \leq \sqrt{x}$. Then uniformly over all sets of primes $Q$ satisfying \eqref{definition of delta} with $\delta \ge c$, we have 
		$$H_Q(x,y,2y) \ll_{c,\kappa} \frac{x}{(\log x)^{1-\density}} \sum_{a \in \mathscr{P}(2y) \cap \S}\frac{L(a)}{a \log(y^{3/20}/a + P^+(a))^{1+\delta}}.$$
	\end{lemma}
	\Cref{lem UB 1} is proven in \Cref{sec upper bound proofs}.
		\begin{lemma}\label{lem UB 3} Let $c,\kappa,y_{0}$ be as in \Cref{mainthm}. Suppose $y_0 < y \leq \sqrt{x}$. Then uniformly over all sets of primes $Q$ satisfying \eqref{definition of delta} with $\delta \ge c$, we have 
			$$\sum_{a \in \mathscr{P}(2y) \cap \S}\frac{L(a)}{a \log(y^{3/20}/a+P^+(a) )^{1+\density}} \ll_c \frac{1}{(\log y)^{1+\density}} \sum_{a \in \mathscr{P}(2y) \cap \S}\frac{L(a)}{a}.$$
		\end{lemma}
		\begin{proof}
			We apply \Cref{dimitris lemma 2.2} with $f(a) = \bm{1}_{\S}(a) \cdot L(a), h = 1+\delta,$ and $\epsilon= 3/20.$ The conditions of the lemma are met, since for any $a,p$ such that $(a,p) = 1$, we have by \Cref{lemma L properties} (ii) that $f(ap) \leq 2 f(a)$. 
		\end{proof}
	Combining \Cref{lem UB 1} and \Cref{lem UB 3} gives an upper bound of the form \begin{equation} \label{UB eq H to L}
			H_Q(x,y,2y) \ll_{c,\kappa} \frac{x}{(\log x)^{1-\density}(\log y)^{1+\density}} \sum_{a \in \mathscr{P}(2y) \cap \S} \frac{L(a)}{a}.
		\end{equation}
		We will cut up the above sum according to the number of prime factors of $a$. Hence, we desire an upper bound for
		$$T_Q(k,2y) := \sum_{\substack{a \in \mathscr{P}(2y) \cap \S \\ \omega(a) = k}} \frac{L(a)}{a}, \quad (k \in \N).$$
		The next lemma is the first step in this direction.
		\begin{lemma} \label{lem UB 4}
			Let $c,\kappa$ be as in \Cref{mainthm}. Suppose $y$ is sufficiently large in terms of $c,\kappa$. Let $v = \lfloor\frac{1}{\log 2}\log \log y\rfloor$, and suppose that $1 \leq k \leq 10v$. One has
			$$T_Q(k,2y) \ll (2 \density \log_2 y + O(\kappa))^k U_k(v),$$
			where 
			\begin{equation}\label{def Uk}
			U_k(v) = \underset{{0 \leq \xi_1 \leq \cdots \leq \xi_k \leq 1}}{\idotsint} \min_{0 \leq j \leq k} 2^{-j}(2^{v \xi_1} + \cdots + 2^{v \xi_j} +1)d\bm{\xi}.	
			\end{equation} 
		\end{lemma}
		\Cref{lem UB 4} is proven in \Cref{sec upper bound proofs}. The last step is to provide an estimate for $U_k(v)$. This is the content of the next lemma.
		\begin{lemma}\label{lem UB 5} Let $k,v$ be integers such that  $1 \leq k \leq 10v$. One has
			$$U_k(v) \ll \frac{1 + |v-k|}{(k+1)! (2^{(k-v)/2}+1)}.$$
		\end{lemma}
		\begin{proof}
			This follows from Lemma 3.6 \cite{ford2006integers} combined with the refined Lemma 3.3.3 from \cite{koukoulopoulos2010generalized}. We give a sketch of the proof here, citing the relevant results. Let $\gamma \geq 0$ be given, and define
			$$\mathcal{T}(k,v,\gamma) := \{\bm{\xi} \in \R^k : 0 \leq \xi_1 \leq \cdots \leq \xi_k \leq 1, 2^{v\xi_1}+\cdots + 2^{v\xi_j} \geq 2^{j-\gamma} \; (1 \leq j \leq k)\}.$$
			\noindent Lemma 3.3.3 from \cite{koukoulopoulos2010generalized} states that for $\gamma \geq 0, v \geq 1,$ and $\gamma + v - k \geq -C,$ where $C \geq 0$ is a constant, one has
			\begin{equation*}
				\Vol\big(\mathcal{T}(k,v,\gamma)\big) \ll_C \frac{(\gamma + 2)(v-k+\gamma + C + 1)}{(k+1)!}.
			\end{equation*}
			If we assume $1 \le k \leq v,$ then we can take $C=0$ above, and the bound becomes
			\begin{equation}\label{Dimitris T bound}
				\Vol\big(\mathcal{T}(k,v,\gamma)\big) \ll \frac{(\gamma + 2)(v-k+\gamma + 1)}{(k+1)!}, \quad (1 \leq k \leq v).
			\end{equation}
			As is noted in the proof of Lemma 3.6 in \cite{ford2006integers}, we have
			$$U_k(v) \leq \sum_{m \geq 0 }2^{1-m}\Vol\big(\mathcal{T}(k,v,m+1)\big).$$
			By substituting \eqref{Dimitris T bound} into the above, we find that for $1 \leq k \leq v$, one has
			\begin{align*}
				U_k(v) &\ll \frac{1}{(k+1)!}\sum_{m \geq 0 }2^{-m}(m+ 3)(v-k+m+2)\\& \asymp \frac{v-k+1}{(k+1)!}+ \frac{1}{(k+1)!}\sum_{m \geq 1}2^{-m}(m+ 3)(v-k+m+2).
			\end{align*}
			In the $m$ sum on the right hand side of the above equation, the ratio of the $(m+1)^{\text{st}}$ term with the $m^{\text{th}}$  term is 
			$$\frac{1}{2} \cdot \frac{(v-k+m+3)(m+4)}{(v-k+m+2)(m+3)} \leq \frac{1}{2} \cdot \frac{(m+4)}{(m+2)} < 1,$$
			since $v-k \geq 0$. As such, the $m$ sum above is dominated by the first term $m=1$, and we conclude that one has
			\begin{align}\label{U bound Dimitris}
				U_k(v) &\ll \frac{v-k+1}{(k+1)!}, \quad (1 \leq k \leq v).
			\end{align}
			On the other hand, it is shown in Lemma 3.6 of \cite{ford2006integers} that
			\begin{align*}
				U_k(v) &\ll \frac{|v-k|^2+1}{(k+1)!2^{k-v}}, \quad (v \leq k \leq 10v).
			\end{align*}
			Since $t+1 \ll 2^{t/2}$ for $t \geq 0$, the above bound can be rewritten as
			\begin{align}\label{U bound Ford}
				U_k(v) &\ll \frac{|v-k|+1}{(k+1)!2^{(k-v)/2}}, \quad (v \leq k \leq 10v).
			\end{align}
			Since $2^{(k-v)/2} \ll 1$ when $k \le v$, we can combine \eqref{U bound Dimitris} and \eqref{U bound Ford} into a single estimate:
			$$U_k(v) \ll \frac{|v-k|+1}{(k+1)!(2^{(k-v)/2}+1)}, \quad (1 \leq k \leq 10v).$$
			This completes the proof of the lemma.\end{proof}

		We are now ready to complete the upper bound proof. We distinguish three ranges of $k$. The first one is $k > 10v$. Using bound (i) from \Cref{lemma L properties}, we have $L(a)\le \tau(a)=2^k$ for every integer $a$ in the summation range of $T_Q(k,2y)$ (recall all these integers are square-free). Thus,
		\begin{equation*}
			\sum_{k > 10v} T_Q(k,2y) \le \sum_{k > 10v}  2^k \sum_{\substack{a \in \mathscr{P}(2y) \cap \S \\ \omega(a) = k}} \frac{1}{a}.
		\end{equation*}
		Note that 
		$$\sum_{\substack{a \in \mathscr{P}(2y) \cap \S \\ \omega(a) = k}} \frac{1}{a} \le \frac{1}{k!}\sum_{\substack{p_1,\dots,p_k \in [1,2y] \cap \S\\ p_i \ne p_j }} \frac{1}{p_1 \cdots p_k} \leq \frac{1}{k!}\bigg(\sum_{\substack{p\in [1,2y] \cap \S}} \frac{1}{p}\bigg)^k.$$
		Using \Cref{defining Q}, we thus have
		\begin{align*}
			\sum_{k > 10v} T_Q(k,2y) &\le \sum_{k > 10v}  \frac{2^k }{k!}\bigg(\sum_{\substack{p\in [1,2y] \cap \S}} \frac{1}{p}\bigg)^k\\
			&= \sum_{k > 10v}  \frac{\big( 2\delta \log \log y + O(\kappa)\big)^k}{k!}.
		\end{align*}
	Since $10v > 10 \log_2 y > 2\density \log_2 y + O(\kappa)$ for sufficiently large $y$, then the above Poisson sum is dominated by the single term $k = 10v+1$. Thus, we have
	\begin{align*}
		\sum_{k > 10v} T_Q(k,2y) &\ll \frac{\left(2\density \log_2 y +O(\kappa) \right)^{10v+1}}{(10v+1)!}\\
		&\ll_\kappa \frac{\left(2\density \log_2 y\right)^{v}}{(v+1)!}.
	\end{align*}
	Our second range is $v < k \leq 10v$. Here, we use \Cref{lem UB 4} and \Cref{lem UB 5}:
	\begin{align*}
		\sum_{v < k \leq 10v} T_Q(k,2y) &\ll \sum_{v < k \leq 10v} \frac{k-v+1}{2^{(k-v)/2}}\cdot\frac{(2\density \log_2y + O(\kappa))^k}{(k+1)!}.
	\end{align*}
	By re-indexing the sum with $j = k-v$, we can factor out the $v$-th term. Using the fact that $\frac{2\density \log_2 y}{v} = \density \log 4 + O(1/v)$, we have
	\begin{align*}
		\sum_{v < k \leq 10v} T_Q(k,2y) &\ll \frac{(2\density \log_2y + O(\kappa))^v}{(v+1)!} \sum_{j=1}^{\infty} \frac{j+1}{2^{j/2}} \bigg( \frac{2\density \log_2 y + O(\kappa)}{v+2} \bigg)^j\\
		&\ll \frac{(2\density \log_2y + O(\kappa))^v}{(v+1)!} \sum_{j=1}^{\infty} (j+1) \bigg( \frac{\density \log 4}{\sqrt{2}} + o(1) \bigg)^j.
	\end{align*}
	Because $\density \leq 1$ and $\frac{\log 4}{\sqrt{2}} \approx 0.98 < 1$, the infinite series can be bounded by an absolute constant. We conclude
	\begin{align*}
		\sum_{v < k \leq 10v} T_Q(k,2y) \ll \frac{(2\density \log_2y + O(\kappa))^v}{(v+1)!} \ll_\kappa \frac{(2\density \log_2y)^v}{(v+1)!}.
	\end{align*}
	The final range is $0 \leq k \leq v$. We again apply  \Cref{lem UB 4} and \Cref{lem UB 5} to get
	\begin{align*}
		\sum_{0 \leq k \leq v} T_Q(k,2y) &\ll \sum_{1 \leq k \leq v} \frac{v-k+1}{(k+1)! (2^{(k-v)/2}+1)}\cdot (2\density \log_2y + O(\kappa))^k\\
		&\asymp \sum_{1 \leq k \leq v} \frac{v-k+1}{(k+1)!}\cdot (2\density \log_2y+O(\kappa))^k.
	\end{align*}
	We wish to replace $(k+1)!$ with $v \cdot k!$ in the denominator of the above sum. This can be done by noting that the contribution of $k \leq v/100$ is negligible. Indeed, the sum
	$$\sum_{1 \leq k \leq v/100 - 1} \frac{v-k+1}{(k+1)!}\cdot (2\density \log_2y +O(\kappa))^k $$
	is, by the ratio test, dominated by the single term $k = v/100 - 1$. As such, we have
	\begin{align*}
		\sum_{0 \leq k \leq v} T_Q(k,2y) &= \sum_{v/100 \leq k \leq v} \frac{v-k+1}{v}\cdot \frac{(2\density \log_2y+O(\kappa))^k}{k!} +O\Big(\frac{(99v + 2)\cdot(2 \delta \log \log y)^{v/100 - 1}}{100 (v/100)!}\Big)\\
		&\ll \sum_{v/100 \leq k \leq v} \frac{v-k+1}{v}\cdot \frac{(2\density \log_2y+O(\kappa))^k}{k!}\\
		&\ll \sum_{1 \leq k \leq v} \frac{v-k+1}{v}\cdot \frac{(2\density \log_2y+O(\kappa))^k}{k!},
	\end{align*}
	since the error term in the first line is exponentially smaller than the term $k = v$. Putting together the bounds from our three cases gives
	$$H_Q(x,y,2y) \ll_{c,\kappa} \frac{x}{(\log x)^{1-\density}(\log y)^{1+\density}}\sum_{1 \leq k \leq v} \frac{v-k+1}{v}\cdot\frac{(2\density \log_2y)^k}{k!}.$$
	This establishes the upper bound for \Cref{prop 1.4}. Combining this with the lower bound \eqref{final LB} established in \Cref{sec lower bounds} yields the asymptotic order of magnitude claimed in \eqref{main lemma}. This completes the proof of \Cref{prop 1.4}.
	\section{Proof of \Cref{prop 1.5}: Phase Transitions for a poisson-type sum}\label{sec trans range}
	In this section, we study the behavior of the sum $\sum_{1 \leq k \leq v} \frac{\lambda^k}{k!}\cdot\frac{(v-k+1)}{v}$ for varying $\lambda.$ In particular, we demonstrate a change of behaviour for the sum with the critical point being $\lambda = v$. Before we begin, we state one useful lemma on the partial sums of the Poisson distribution. The proof follows from the saddle-point expansion for the Poisson distribution; see \cite[Example 2.4.5]{jensen1995saddlepoint} for the general asymptotic expansion from which this leading-order term is derived.
	\begin{lemma}\label{lem: saddlepoint poisson}
		Let $\lambda$ be a given number, and let $z = z(\lambda)$ be such that $z = o(\lambda^{2/3}).$ Then
		$$\sum_{k \leq \lambda + z} \frac{\lambda^k}{k!} \sim \frac{e^\lambda}{\sqrt{2\pi}} \int_{-\infty}^{z/\sqrt{\lambda}}e^{-t^2/2}dt \quad (\lambda \to \infty).$$
	\end{lemma}
	We are now ready to prove the main result of this section.
	\begin{proof}[Proof of \Cref{prop 1.5}]
		Fix $\epsilon > 0$. For convenience throughout the proof, we define the parameter $\theta := \lambda - v$. Since $\lambda = 2\density \log \log y$ and $v = \lfloor\frac{1}{\log 2}\log \log y\rfloor$, we note the relation $\theta \asymp (\density - \frac{1}{\log 4})\log \log y$. Furthermore, we will write
		$$a_k := \frac{\lambda^k}{k!}\cdot \frac{v-k+1}{v}.$$
		
		To establish the bounds in the proposition, it is convenient to divide the range of $\theta$ into five distinct regimes. We will prove the following five estimates, which collectively imply the proposition:
		\begin{enumerate}[label=(\roman*)]
			\item If $\density \in \big(\epsilon,\frac{1}{\log 4} - \epsilon\big)$ (so that $\theta \leq 0$ and $|\theta| \gg_\epsilon \lambda$), then $$\sum_{1 \leq k \leq v} a_k \asymp_{\epsilon} e^\lambda.$$
			\item If $\density \in \big(\frac{1}{\log 4} + \epsilon,1\big]$ (so that $\theta > 0$ and $\theta \gg_\epsilon \lambda$), then $$\sum_{1 \leq k \leq v} a_k \asymp_{\epsilon} \frac{\lambda^v}{(v+1)!}.$$
			\item If $|\theta|  \le \sqrt{\lambda},$ then $$\sum_{1 \leq k \leq v} a_k \asymp \frac{\lambda^v}{v!} \asymp \frac{e^\lambda}{\sqrt{\lambda}}.$$
			\item If $\theta \leq -\sqrt{\lambda}$ and $\theta = o(\lambda),$ then $$\sum_{1 \leq k \leq v} a_k \asymp \frac{|\theta| e^\lambda}{\lambda}.$$
			\item If $\theta \geq \sqrt{\lambda}$ and $\theta = o(\lambda),$ then $$\sum_{1 \leq k \leq v} a_k \asymp \frac{\lambda^v}{v!} \cdot \frac{v}{\theta^2}.$$
		\end{enumerate}
	We have
	\begin{align*}
		\sum_{1 \leq k \leq v}a_k &=\bigg(1+\frac{1}{v}\bigg)\sum_{1 \leq k \leq v} \frac{\lambda^k}{k!} - \frac{1}{v}\sum_{1 \leq k \leq v}\frac{k\cdot\lambda^{k}}{k!}\\ &=\bigg(1+\frac{1}{v}\bigg)\sum_{1 \leq k \leq v} \frac{\lambda^k}{k!} - \frac{\lambda}{v}\sum_{1 \leq k \leq v}\frac{\lambda^{k-1}}{(k-1)!}\\
		&= \bigg(\frac{v-\lambda+1}{v}\bigg)\sum_{1 \leq k \leq v} \frac{\lambda^k}{k!} + \frac{\lambda}{v}\bigg(\frac{\lambda^{v}}{v!}-1\bigg).
	\end{align*}
	Since $v-\lambda = - \theta$ and $\lambda/v \asymp 1,$ the above calculation implies that
	\begin{align}\label{poisson key identity}
		\sum_{1 \leq k \leq v}a_k &= \bigg(\frac{1-\theta}{v}\bigg)\sum_{1 \leq k \leq  v}\frac{\lambda^k}{k!} + \frac{\lambda^{v+1}}{v\cdot v!} +O(1).
	\end{align}
	\Cref{poisson key identity} is a key identity which will be used throughout the proof. We now begin proving the different parts of the Lemma. 
	\medskip\newline\noindent\textit{Proof of (i)}. Here, we have $- \theta \geq 0$ and $-\theta \asymp_\epsilon v.$ Thus, from \eqref{poisson key identity}, we have
	\begin{align}\label{poisson lemma i eq 1}
		\sum_{1\le k\le v}a_k \asymp \sum_{1\le k\le v}\frac{\lambda^k}{k!} + \frac{\lambda^{v+1}}{(v+1)!}+O(1) =  \sum_{k=1}^{v+1} \frac{\lambda^k}{k!}+O(1).
	\end{align}
	On the one hand, we have
	$$ \sum_{1 \leq k \leq v+1} \frac{\lambda^k}{k!} \leq \sum_{0 \leq j \leq \infty}\frac{\lambda^k}{k!} = e^\lambda.$$
	This proves the upper bound for (i). On the other hand, since $\delta < 1/\log 4$ we have $\lambda < v$. Thus, \begin{equation}\label{eq cite for part iv 1}
		\sum_{1\le k\le v+1}\lambda^k/k!  \ge \sum_{1\le k\le \lambda}\lambda^k/k!.
	\end{equation}
	By \Cref{lem: saddlepoint poisson}, we have
	\begin{equation}\label{eq cite for part iv 2}
		\sum_{1\le k\le \lambda}\frac{\lambda^k}{k!} \sim \frac{e^\lambda}{\sqrt{2 \pi}} \int_{-\infty}^{0}e^{-t^2/2}dt = \frac{e^{\lambda}}{2}.
	\end{equation}
	This proves the lower bound in (i) and hence concludes the proof of this part.
	\noindent\medskip\newline\textit{Proof of (ii)}. In the case when $\density > \frac{1}{\log 4}+\epsilon$, we have $(\delta \log 4)^{-1} \leq 1-\epsilon/2$. Furthermore, since $v = \frac{1}{\log 2} \log \log y + O(1)$, then for sufficiently large $y$, we have 
	\begin{equation}\label{eq poisson lemma v/lambda}
		\frac{v}{\lambda} = \frac{1}{\delta \log 4} + \frac{\epsilon}{4} \leq 1 - \frac{\epsilon}{4}.
	\end{equation}
	By making the change of index $j = v-k,$ we see that
	$$\sum_{1 \leq k \leq v} a_k = \sum_{0 \leq j \leq v-1}a_{v-j} = \sum_{0 \leq j \leq v-1}\frac{j+1}{v}\cdot \frac{\lambda^{v-j}}{(v-j)!}.$$
	As such,
	$$\frac{1}{a_v}\sum_{1 \leq k \leq v} a_k = \sum_{0 \leq j \leq v-1} (j+1) \lambda^{-j} \frac{v!}{(v-j)!} \leq \sum_{0 \leq j \leq v-1} (j+1) \Big(\frac{v}{\lambda}\Big)^j,$$
	having used the inequality $v!/(v-j)! \leq v^j$. Substituting \eqref{eq poisson lemma v/lambda} into the above, we see that
	\begin{equation}\label{eq poisson part ii 1}
		\frac{1}{a_v}\sum_{1 \leq k \leq v} a_k \leq \sum_{0 \leq j \leq v-1} (j+1) \cdot \big(1-\epsilon/4 \big)^j \ll_\epsilon 1,
	\end{equation}
	where we've handled the $j$-sum above using a classical summation identity for geometric series:
	$$\sum_{0 \leq j \leq v+1}(j+1) r^j \leq \sum_{1 \leq j \leq \infty}j r^{j-1} = \frac{1}{(1-r)^2}, \quad (|r| < 1).$$
	Multiplying through by $a_v$ on both sides of \eqref{eq poisson part ii 1} shows that
	\begin{equation}\label{eq poisson part ii 2}
		\sum_{1 \leq k \leq v} a_k \ll_\epsilon a_v.
	\end{equation}
	On the other hand, since $a_k > 0$ for $1 \leq k \leq v$, we have
	\begin{equation}\label{eq poisson part ii 3}
		\sum_{1 \leq k \leq v} a_k \ge   a_v \asymp \frac{\lambda^{v}}{(v+1)!}.
	\end{equation}
	\Cref{eq poisson part ii 1} together with \eqref{eq poisson part ii 3} completes the proof of (ii).
	
	\medskip\noindent\textit{Proof of (iii).} It is helpful to treat the cases $\theta \leq 1/2$ and $\theta > 1/2$ separately. Let us first suppose $\theta \leq 1/2$. In this case, $(1-\theta)/v \geq 0$, and so \eqref{poisson key identity} is a sum of two non-negative terms (and a $O(1)$ error term). The first of these two terms is 
	\begin{equation}\label{poisson lemma iii eq 1.1}
		\bigg(\frac{1-\theta}{v}\bigg) \sum_{1 \leq k \leq v} \frac{\lambda^k}{k!} \asymp \bigg(\frac{1+|\theta|}{v}\bigg)\sum_{1 \leq k \leq v} \frac{\lambda^k}{k!}.
	\end{equation}
	Since $|\theta |\ll \sqrt{\lambda}$, we have by \Cref{lem: saddlepoint poisson} that
	$$\sum_{1 \leq k \leq v} \frac{\lambda^k}{k!} =\sum_{1 \leq k \leq \lambda - \theta} \frac{\lambda^k}{k!} \sim \frac{e^{\lambda}}{\sqrt{2 \pi}}\int_{- \infty}^{-\theta/\sqrt{\lambda}} e^{-t^2/2}dt.$$
	We see that \begin{equation*}
		\frac{1}{2} = \frac{1}{\sqrt{2 \pi}}\int_{- \infty}^{0} e^{-t^2/2}dt \leq \frac{1}{\sqrt{2 \pi}}\int_{- \infty}^{-\theta/\sqrt{\lambda}} e^{-t^2/2}dt \leq \frac{1}{\sqrt{2 \pi}}\int_{- \infty}^{\infty} e^{-t^2/2}dt = 1,
	\end{equation*}
	and so
	$$\sum_{1 \leq k \leq v} \frac{\lambda^k}{k!} \asymp e^\lambda.$$ Substituting this back into \eqref{poisson lemma iii eq 1.1}, and using the fact that $|\theta| \ll \lambda$ proves that 
	\begin{equation}
		\label{eq: poisson 11}
		\bigg(\frac{1-\theta}{v}\bigg) \sum_{1 \leq k \leq v} \frac{\lambda^k}{k!} \asymp \bigg(\frac{1+|\theta|}{v}\bigg) \cdot e^\lambda.
	\end{equation}
	The second term in \eqref{poisson key identity} is
	\begin{align}
		\frac{\lambda^{v+1}}{(v+1)!} \asymp \frac{\lambda^{v}}{v!} = \frac{v^v}{v!}\bigg(\frac{\lambda}{v}\bigg)^v &\asymp \frac{e^v}{\sqrt{v}}\bigg(1-\frac{\theta}{\lambda}\bigg)^{-\lambda \frac{v}{\lambda}},\nonumber
	\end{align}
	By Stirling's formula. The assumption $|\theta| \le \sqrt{\lambda}$ implies that $(1-\theta/\lambda)^{-\lambda} \asymp e^\theta$. In addition, $v/\lambda = (1-\theta/\lambda),$ hence
	\begin{align}\label{poisson lemma iii eq 2}
		\frac{\lambda^{v+1}}{(v+1)!} \asymp \frac{e^v}{\sqrt{\lambda}}e^{\theta(1-\frac{\theta}{\lambda})} \asymp \frac{e^{\lambda - \frac{\theta^2}{\lambda}}}{\sqrt{\lambda}} \asymp \frac{e^{\lambda}}{\sqrt{\lambda}},
	\end{align}
	again using the assumption $|\theta| \le \sqrt{\lambda}$. Substituting \eqref{eq: poisson 11} and \eqref{poisson lemma iii eq 2} into \eqref{poisson key identity} implies that
	$$\sum_{1 \leq k \leq v}\frac{\lambda^k}{k!}\cdot\frac{(v-k+1)}{v} \asymp \bigg(\frac{1 + |\theta| }{v} + \frac{1}{\sqrt{\lambda}}\bigg)\cdot e^\lambda \asymp \frac{e^\lambda}{\sqrt{\lambda}},$$
	since $|\theta| \ll \sqrt{\lambda}$. This concludes the proof of (iii) in the case when $\theta \leq 1/2$. \par Now suppose $\theta > 1/2$. In this case, the right-hand side of \eqref{poisson key identity} is a sum of one positive and one possibly negative term which are potentially of the same size. Thus, we must be careful of potential cancellations. By the triangle inequality and \eqref{poisson key identity}, we have
	\begin{equation}\label{poisson lemma iii eq 4}
		\bigg|\sum_{1 \leq k \leq v}a_k \bigg| \leq \bigg(\frac{1+|\theta|}{v}\bigg)\sum_{1 \leq k \leq  v}\frac{\lambda^k}{k!} + \frac{\lambda^{v+1}}{v\cdot v!} +O(1) \ll \frac{e^\lambda}{\sqrt{\lambda}},
	\end{equation}
	with the last inequality following by the same argument we used in the case when $\theta\le 1/2$ above.
	It remains to prove a matching lower bound. Note that $$a_k = \frac{v-k+1}{v} \cdot \frac{\lambda^k}{k!} \geq \frac{1}{v} > 0.$$
	Furthermore, since $\theta \leq \sqrt{\lambda}$ we have
	\begin{equation} \label{poisson lemma iii eq 5}
		\sum_{1 \leq k \leq v}a_k = \sum_{1 \leq k \leq \lambda -\theta}a_k \geq \sum_{1 \leq k \leq \lambda -\sqrt{\lambda}} a_k.
	\end{equation}
	\noindent By part (v) of our proof (which we prove below),\footnote{The proof of (v) appears after that of (iii) for aesthetic reasons. The reader can verify that there is no circular reasoning, as the proof of (v) does not use the claim (iii) in any way.} we have
	\begin{equation}\label{poisson lemma iii eq 6}
		\sum_{1 \leq k \leq \lambda - \sqrt{\lambda}} a_k \gg \frac{\lambda^v}{v!}\cdot \frac{v}{\lambda}.
	\end{equation}
	By \eqref{poisson lemma iii eq 5} and \eqref{poisson lemma iii eq 6} we have 
	\begin{equation}\label{poisson lemma iii eq 7}
		\sum_{1 \leq k \leq v} a_k \gg \frac{e^\lambda}{\sqrt{\lambda}}.
	\end{equation}
	Putting together the matching upper bound of \eqref{poisson lemma iii eq 4} and the lower bound of \eqref{poisson lemma iii eq 7} proves (iii) in this second case. 
	
	\medskip\noindent\textit{Proof of (iv).} If $\theta < 0$, then $v \geq \lambda$. Arguing as in \eqref{eq cite for part iv 1} and \eqref{eq cite for part iv 2}, we have
	$$\sum_{1 \leq k \leq v}\frac{\lambda^k}{k!} \asymp e^\lambda.$$
	Substituting this into \eqref{poisson key identity} and using the fact that $(1-\theta)/v > 0$, we have
	\begin{align}
		\sum_{1 \leq k \leq v}\frac{\lambda^k}{k!}\cdot\frac{(v-k+1)}{v} &\asymp \bigg(\frac{1-\theta}{v}\bigg)\sum_{1 \leq k \leq  v}\frac{\lambda^k}{k!} + \frac{\lambda^{v+1}}{v\cdot v!} + O(1)\nonumber\\
		&\asymp \frac{\lambda^{v+1}}{v\cdot v!} + \frac{e^\lambda(1+|\theta|)}{v} +O(1). \label{pois est iv 1}
	\end{align}
	Since $\lambda/v \asymp 1,$ we have	\begin{align}
		\frac{\lambda^{v+1}}{v\cdot v!} &\asymp \frac{\lambda^v}{v!}.
	\end{align}
	By Stirling's approximation,
	$$\frac{\lambda^v}{v!} \asymp \frac{e^v \lambda^v}{v^v \sqrt{v}} = \frac{e^v }{\sqrt{v}} \cdot \exp\big(v \log (\lambda/v)\big) = \frac{e^v }{\sqrt{v}} \cdot \exp\big(v \log (1+\theta/v)\big),$$
	having written $\lambda = v+\theta.$ Using the inequality $\log(1+u) \leq u - u^2/4,$ which is valid for $u = \theta/v = o(1)$, the above becomes
	$$\frac{\lambda^{v+1}}{v\cdot v!} \ll \frac{e^v }{\sqrt{v}} \cdot \exp\bigg(\theta - \frac{\theta^2}{4v}\bigg) = \frac{e^{v+\theta} }{\sqrt{v}} \cdot \exp\bigg(- \frac{\theta^2}{4v}\bigg) = \frac{e^\lambda }{\sqrt{v}} \cdot \exp\bigg( - \frac{\theta^2}{4v}\bigg).$$
	Since we've assumed that $|\theta| \ge \sqrt{\lambda}$, then $e^{- \frac{\theta^2}{4v}} \ll |\theta|/\sqrt{\lambda},$ and so we conclude that
	$$\frac{\lambda^{v+1}}{v\cdot v!} \ll \frac{e^\lambda }{\sqrt{v}} \cdot \frac{|\theta|}{\sqrt{v}} \asymp \frac{|\theta| e^\lambda}{v}.$$
	Substituting this into \eqref{pois est iv 1} gives
	\begin{align*}
		\sum_{1 \leq k \leq v}\frac{\lambda^k}{k!}\cdot\frac{(v-k+1)}{v} &\asymp \frac{e^\lambda|\theta|}{v}.
	\end{align*}
	This completes the proof of (iv).
	
	\medskip\noindent\textit{Proof of (v).} First, note the trivial bound
	$$ \sum_{1 \leq k \le v/100} \frac{\lambda^k (v-k+1)}{k! \cdot v} \ll \exp[v/10], $$
	which is considerably smaller than the claimed size of the full sum over $k$. Let us now consider the range $k \in (v/100,v]$. For each $k$, define $h_k$ to be the smallest non-negative integer such that 
	$$ \frac{\lambda^{k-h_k}}{(k-h_k)!} \leq \frac{1}{2} \frac{\lambda^k}{k!}. $$
	The number $h_k$ exists for each $k \in [v/100,v]$ since it is the minimum of a nonempty finite set. Indeed, $\lambda^0/0! < (1/2) \lambda^k/k!$ for each $k \in [v/100,v]$, and
	$$ h_k = \min \bigg\{n \in \mathbb{Z}_{\geq 0}: \frac{\lambda^{k-n}}{(k-n)!} \leq \frac{1}{2} \frac{\lambda^k}{k!}\bigg\}. $$
	Having proven the existence of $h_k$, our next goal is to identify its size. 
	
	\begin{claim} For $k \in [v/100,v]$, we have
		\begin{equation}\label{hv closed form}
			h_k \asymp \frac{1}{\log(\lambda/k)}.
		\end{equation}
	\end{claim}
	\begin{proof}[Proof of claim]
		Let us first show that $h_k \ll 1/\log(\lambda/k)$. Let $\beta > 0$ be a parameter to be chosen later. We note that $h_k \leq \lfloor \beta/ \log (\lambda/k) \rfloor$ is equivalent to 
		\begin{equation}\label{poisson lemma v eq a}
			\lambda^{-\lfloor \beta/ \log (\lambda/k) \rfloor} \bigg(\frac{k!}{(k-\lfloor \beta/ \log (\lambda/k) \rfloor)!}\bigg) \leq \frac{1}{2}.
		\end{equation}
		Since $k!/(k-j!) \leq k^j$, we have 
		$$ \lambda^{-\lfloor \beta/ \log (\lambda/k) \rfloor} \bigg(\frac{k!}{(k-\lfloor \beta/ \log (\lambda/k) \rfloor)!}\bigg) \leq (k/\lambda)^{\lfloor \beta/ \log (\lambda/k) \rfloor} = e^{-\beta} \big(\lambda/k \big)^{f}, $$
		where $f = \{\beta/ \log (\lambda/k)\}$, and $\{\cdot\}$ denotes the fractional part. Since $f \in [0,1)$ and $k \leq \lambda$ by the assumption of (v), we have $\big(\lambda/k\big)^{f} \leq \lambda/k$. Thus,
		$$ \lambda^{-\lfloor \beta/ \log (\lambda/k) \rfloor} \bigg(\frac{k!}{(k-\lfloor \beta/ \log (\lambda/k) \rfloor)!}\bigg) \leq e^{-\beta} \frac{\lambda}{k}. $$
		By the assumption $k \geq v/100$, we have
		$$ e^{-\beta}\frac{\lambda}{k} \leq \exp\big[\log 100-\beta - \log(v/\lambda)\big] = \exp\big[\log 100 - \beta - \log(1-\theta/\lambda)\big]. $$
		Since $\theta = o(\lambda)$, then for sufficiently large $\lambda$ we have $- \log(1-\theta/\lambda) \leq 1$. The conclusion is that
		$$ \lambda^{-\lfloor \beta/ \log (\lambda/k) \rfloor} \bigg(\frac{k!}{(k-\lfloor \beta/ \log (\lambda/k) \rfloor)!}\bigg) \leq \exp\big[1+\log 100-\beta\big]. $$
		By choosing $\beta$ sufficiently large, say $\beta = 10$, we see that the above implies \eqref{poisson lemma v eq a}, and hence $h_k \ll 1/\log(\lambda/k)$. 
		
		Let us now show that $h_k \gg 1/\log(\lambda/k)$. Let $\gamma \in (0,1]$ be a parameter to be chosen later. Because of the minimality condition in the definition of $h_k$, we note that the claim $h_k > \lfloor \gamma/ \log (\lambda/k) \rfloor$ is equivalent to
		\begin{equation}\label{poisson lemma v eq b}
			\lambda^{-\lfloor \gamma/ \log (\lambda/k) \rfloor} \bigg(\frac{k!}{(k-\lfloor \gamma/ \log (\lambda/k) \rfloor)!}\bigg) > \frac{1}{2}.
		\end{equation}
		Write $x = \lfloor \gamma/ \log (\lambda/k) \rfloor$, and note that $0 < x \ll \gamma \sqrt{\lambda}$ by our assumption that $\theta \geq \sqrt{\lambda}$. Indeed, since $\log(\lambda/k)$ is minimized on our interval at $k=v$, and $\log(\lambda/v) = \log(1+\theta/v) \geq \theta/(2v)$, we have
		$$ x \leq \frac{\gamma}{\log(\lambda/v)} \leq \frac{2\gamma v}{\theta} \ll \gamma \sqrt{\lambda}. $$
		Using our assumption $k \geq v/100$, we therefore have 
		\begin{align}
			\lambda^{x} \bigg(\frac{(k-x)!}{k!}\bigg) \leq \bigg(\frac{\lambda}{k}\bigg)^{x} \cdot \Big( 1-\frac{x}{k} \Big)^{-x} & = \bigg(\frac{\lambda}{k}\bigg)^{x} \cdot \exp \bigg(\frac{x^2}{k} + O\Big(\frac{x^3}{k^2}\Big)\bigg) \nonumber\\
			&\le \exp\Big(\gamma + O(\gamma^2)\Big).  \label{pois lem v eq claim 3}
		\end{align}
		Choosing $\gamma$ to be a sufficiently small positive real number, we can make the right hand side of \eqref{pois lem v eq claim 3} to be $<2$, i.e., 
		$$ \lambda^{x} \bigg(\frac{(k-x)!}{k!}\bigg) < 2. $$ 
		Hence, \eqref{poisson lemma v eq b} is satisfied with such a $\gamma$, and so $h_k \geq \lfloor \gamma / \log(\lambda/k) \rfloor \gg 1/\log(\lambda/k)$. This completes the proof of the claim.
	\end{proof}
	
	Now, write $v = v_0$, and define recursively $v_{j+1} = v_{j}-h_{v_j}$ for $j \geq 0$. We stop this sequence at step $J$, where $v_{J+1} \leq v/100$.
	\begin{claim}
		For $j \geq 0$ and $k \in (v_{j+1}, v_j]$, we have 
		$$ \frac{\lambda^k}{k!} \leq \frac{1}{2^j}\frac{\lambda^v}{v!}. $$
	\end{claim}
	\begin{proof}[Proof of claim]
		We proceed by induction on $j$ to show the bounds at the sequence endpoints: $\lambda^{v_j}/v_j!$. The base case $j=0$ is trivial since $v_0 = v$. 
		
		Assume the bounds hold for some $j \ge 0$. By the definition of $h_{v_j}$ and the inductive hypothesis, we have
		\begin{equation} \label{pois lem v eq ind 1}
			\frac{\lambda^{v_{j+1}}}{v_{j+1}!} = \frac{\lambda^{v_j - h_{v_j}}}{(v_j - h_{v_j})!} \leq \frac{1}{2} \frac{\lambda^{v_j}}{v_j!} \leq \frac{1}{2} \bigg( \frac{1}{2^j}\frac{\lambda^v}{v!} \bigg) = \frac{1}{2^{j+1}}\frac{\lambda^v}{v!},
		\end{equation}
		completing the induction. Finally, for any $k \in (v_{j+1}, v_j]$, the sequence $\lambda^k/k!$ is monotonically increasing in $k$ (since $k \le v < \lambda$). Therefore, its value is bounded above by the value at the right endpoint ($v_j$). This proves the claim.
	\end{proof}
	
	We can evaluate the sum of $(v-k+1)$ over a single interval $(v_{j+1},v_j]$ exactly by recognizing it as an arithmetic progression. The sequence has $v_j - v_{j+1} = h_{v_j}$ terms. The first term (when $k = v_{j+1}+1$) is $v - v_{j+1}$, and the last term (when $k = v_j$) is $v - v_j + 1$. Using the standard formula for the sum of an arithmetic progression, we obtain
	\begin{align}
		\sum_{k \in (v_{j+1},v_j]} (v-k+1) \nonumber
		&= \frac{h_{v_{j}}}{2}\bigg[ (v-v_j + 1) + (v-v_{j+1}) \bigg]\\ \nonumber
		&= \frac{h_{v_{j}}}{2}\bigg( (v-v_j) + (v-v_j + h_{v_j}) + 1 \bigg)\\ 
		&= \frac{h_{v_{j}}}{2}\bigg( 2\sum_{i = 0}^{j-1}h_{v_i} + h_{v_j} + 1 \bigg), \label{eq: sum the AP}
	\end{align} 
	for each $j \geq 0$.
	
	By partitioning the range $(v/100, v]$ into these intervals $(v_{j+1}, v_j]$ for $0 \leq j \leq J$, and applying the bound from our second claim, we can establish an upper bound for the total sum. 
	
	To upper bound the interval sums, we note from \eqref{hv closed form} that $h_k \asymp 1/\log(\lambda/k)$ is an increasing function of $k$. Therefore, $h_{v_i} \leq h_v$ for all $i \leq J$. Applying this to \eqref{eq: sum the AP} yields
	$$ \sum_{k \in (v_{j+1},v_j]} (v-k+1) \ll h_{v_j} \sum_{i=0}^j h_{v_i} \leq (j+1) h_v^2. $$
	Dividing by $v$ and applying the second claim gives
	$$ \sum_{v/100 < k \leq v} \frac{\lambda^k}{k!}\cdot \frac{v-k+1}{v} \ll \frac{\lambda^v}{v \cdot v!} \sum_{j=0}^J \frac{j+1}{2^j} h_v^2 \ll \frac{\lambda^v}{v\cdot v!} \cdot h_v^2. $$
	
	For the lower bound, since all terms are positive, we restrict our attention to the interval $j=0$ (i.e., $k \in (v_1, v_0]$) and use the fact that $\lambda^k/k!$ is monotonically increasing in this range:
	\begin{equation}\label{eq: lower bounding interval decomp}
		\sum_{v/100 < k \leq v} \frac{\lambda^k}{k!}\cdot \frac{v-k+1}{v} \geq \frac{\lambda^{v_1}}{v_1!} \sum_{k \in (v_1, v_0]} \frac{v-k+1}{v}.
	\end{equation}
	To bound $\lambda^{v_1}/v_1!$ from below, we use the minimality of $h_{v_0}$. By definition, the term immediately preceding $v_1$ strictly fails the halving condition:
	$$ \frac{\lambda^{v_1 + 1}}{(v_1 + 1)!} > \frac{1}{2} \frac{\lambda^{v_0}}{v_0!} = \frac{1}{2} \frac{\lambda^v}{v!}. $$
	Multiplying by the ratio between consecutive terms gives
	$$ \frac{\lambda^{v_1}}{v_1!} = \frac{\lambda^{v_1 + 1}}{(v_1 + 1)!} \cdot \frac{v_1 + 1}{\lambda} > \frac{1}{2} \frac{\lambda^v}{v!} \cdot \frac{v_1 + 1}{\lambda}. $$
	Since $v_1 = v - h_{v_0}$ and $h_{v_0} \ll \sqrt{v}$, we have $v_1 \sim v \sim \lambda$. Thus, $(v_1+1)/\lambda \gg 1$, which implies
	\begin{equation*}
		\frac{\lambda^{v_1}}{v_1!} \gg \frac{\lambda^{v}}{v!}.
	\end{equation*}
	Substituting this into \eqref{eq: lower bounding interval decomp} gives 
	\begin{equation}\label{eq: lower bounding interval decomp 2}
		\sum_{v/100 < k \leq v} \frac{\lambda^k}{k!}\cdot \frac{v-k+1}{v} \gg \frac{\lambda^{v}}{v \cdot v!} \sum_{k \in (v_1, v_0]} (v-k+1).
	\end{equation}
	By using the closed form evaluated at $j=0$ in \eqref{eq: sum the AP}, we see that 
	$$ \sum_{k \in (v_1, v_0]} (v-k+1) = \frac{h_{v_0}}{2}(h_{v_0}+ 1) \asymp h_{v}^2, $$ 
	since $v_0 = v$. Substituting this into \eqref{eq: lower bounding interval decomp 2} gives
	$$ \sum_{v/100 < k \leq v} \frac{\lambda^k}{k!}\cdot \frac{v-k+1}{v} \gg  \frac{\lambda^{v}}{v \cdot v!} h_v^2. $$
	
	We have proven matching upper and lower bounds for the sum of $k \in (v/100,v]$. Adding the negligible contribution from $k \leq \lambda/100$, and noting that $h_v \asymp 1/\log(\lambda/v) \asymp v/\theta$, we conclude
	$$ \sum_{1 \leq k \leq v} \frac{\lambda^k}{k!}\cdot \frac{v-k+1}{v} \asymp \frac{\lambda^v}{v!} \frac{h_v^2}{v} \asymp \frac{\lambda^v}{v!} \cdot \frac{v}{\theta^2}. $$
	This completes the proof of (v), and hence the proof of the Lemma.
\end{proof}

\section{Strong barriers for Generalized Smirnov Statistics}
The proof of \Cref{lem LB Yk} uses estimates related to uniform order statistics with barrier conditions. In particular, we need upper and lower bounds for $Q_k(u,v)$, which is defined in the following way: If $\xi_1,\dots,\xi_k$ are the order statistics for $k$ random samples from the uniform distribution on $[0,1]$, then 
\begin{equation}\label{def Qk}
	Q_k(u,v) := \P\Big( \xi_i \geq \frac{i-u}{v}, \quad 1 \leq i \leq k \Big).
\end{equation}

The quantity $Q_k(u,v)$ has been referred to as a \textit{generalized Smirnov statistic} in \cite{ford2007generalized} and \cite{Ford2008}. The next lemma is our `go-to' for estimates related to these statistics. Part (i) is due to Daniels \cite{daniels1945statistical}, and part (ii) appears explicitly as Lemma 3.3.1 in \cite{koukoulopoulos2010generalized}, though it essentially follows from Lemma 11.1 of \cite{ford2008distribution} and Theorem 1 of \cite{Ford2008}.
\begin{lemma}\label{order stats lem} Let $Q_k(u,v)$ be defined as in \eqref{def Qk}, assume $k \in \N$, and write $w = u+v-k$.
	\begin{enumerate}[label = (\roman*)]
		\item If $k-v < u \leq 1$, then 
		$$Q_k(u,v) = \frac{w}{v}\bigg(1+\frac{1}{v}\bigg)^{k-1};$$
		\item Uniformly in $u > 0$, $w >0$ and $k$,
		$$Q_k(u,v) \ll \frac{(u+1)(w+1)}{k}.$$
	\end{enumerate}
\end{lemma}
 We now state and prove a lemma on generalized Smirnov statistics with strong barrier conditions. This lemma also simplifies part of the lower bound proof of the main theorem in \cite{ford2008distribution} in the important $z=2y$ case.
\begin{lemma}\label{lemma better smirnov}
	Let $\xi_1,\dots,\xi_k$ be uniform order statistics. Fix $\epsilon \in (0,1/12)$, and $C > 10$. Suppose that $1 \le k \leq v$. Define the events
	\begin{align*}
		B&:= \bigcap_{i = 1}^k\Big\{\xi_i \geq \frac{i-1}{v}\Big\},\\
		\mathscr{B} &:= \bigcap_{i = 1}^k\Bigg\{ \xi_i \geq \max\bigg\{\frac{i-1}{v},\frac{i+\min\{i , (k-i)\}^{1/6-\epsilon}-C}{v}\bigg\}\Bigg\}.
	\end{align*}
	Then,
	$$\P[\mathscr{B}] = \P[B]\big(1 + O\big(C^{-\epsilon}\big)\big).$$
\end{lemma}

\begin{remark}
	One can replace the exponent $1/6-\epsilon$ with the optimal $1/2-\epsilon$ by a more complicated proof using first incidence conditions instead of a union bound. The exponent $1/6-\epsilon$ is sufficient for our purposes of proving \Cref{mainthm}, and so we've decided to keep our simple non-optimal argument. Although there seems to be no literature proving the $1/2-\epsilon$ exponent for the specific case of uniform order statistics, there has been recent work proving an analogous theorem for random walks with Gaussian increments in \cite[Appendix B]{arguin2020fyodorov}.
\end{remark}

\begin{proof}[Proof of \Cref{lemma better smirnov}]
For ease of notation, write $\mu = \frac{1}{6}-\epsilon,$ and define $$M_i := \min \{i , (k-i)\}^{\mu} \quad (1 \leq i \leq k).$$ Note that the event $\mathscr{B}$ can be rewritten as
	\begin{align*}
		\mathscr{B} &:= B \cap \bigg( \bigcap_{i = 1}^k\Big\{ \xi_i \geq \frac{i+M_i-C}{v}\Big\}\bigg).
	\end{align*}
	We define, for any given $u > 0$,
	$$S_k(u,v) := \Big\{\xi \in \R^k : 0 \leq \xi_1 \leq \cdots \leq \xi_k \leq 1;\; \xi_i \geq \frac{i-u}{v}\Big\}.$$
	Note that $\P[B] = k!\Vol(S_k(1,v))$. By a simple union bound, we have
	\begin{equation}\label{bettersmirnov eq 0}
		\P[\mathscr{B} | B] \geq 1 - \sum_{1 \leq j \leq k}\P\big[v\xi_j \leq j + M_j - C \big| B\big].
	\end{equation}
	For fixed $j \in [1,k]$, we want to bound the probability of the event $$B_j=B\cap \{v\xi_j\le j+M_j-C\}.$$ Let $y = y_j = v \xi_j - j.$ When $B_j$ occurs, we have that $y_j\in[-1,M_j-C]$. We also note that $y_j \leq v-j$, since $\xi \leq 1$. We make the change of variables
	\begin{align*}
		\alpha_i &= \xi_i \cdot \frac{v}{y+j} \quad (1 \leq i \leq j-1),\\
		\beta_i &= \bigg( \xi_{j+i} - \frac{y+j}{v} \bigg) \cdot \frac{v}{v-y-j} \quad (1 \leq i \leq k-j).
	\end{align*}
	We see that
	\begin{align*}
		\bm{\alpha} &= (\alpha_1, \dots, \alpha_{j-1}) \in S_{j-1}(1,y+j),\\
		\bm{\beta} &= (\beta_1,\dots,\beta_{k-j}) \in S_{k-j}(y+1,v-y-j),
	\end{align*}
	and $d\bm{\xi} = d\bm{\alpha} \cdot d\bm{\beta} \cdot (v-y-j)^{k-j}(y+j)^{j-1} v^{-k}$. As such, $\P[B_j]$ is bounded above by
	$$\frac{k!}{v^k}\int_{-1}^{Y_j}\frac{(v-y-j)^{k-j}(y+j)^{j-1}}{(k-j)!(j-1)!}Q_{j-1}(1,  y+j)Q_{k-j}(y+1,v-y-j)dy,$$
	where $Y_j := \min\{M_j-C,v-j\}.$
	To bound the $Q$ probabilities, we can use \Cref{order stats lem} (ii). By directly applying this to the above integral,
	\begin{align*}
		\P[B_j]&\ll \frac{k!(v-k+1)}{v^k}\int_{-1}^{Y_j}\frac{(v-y-j)^{k-j}(y+j)^{j-1}}{(k+1-j)!j!} (y+2)^2dy\\
		&\asymp \frac{(v-k+1)}{k \cdot v^k}\int_{-1}^{Y_j}\binom{k+1}{j}(v-y-j)^{k-j}(y+j)^{j-1} (y+2)^2 dy.
	\end{align*}
	From \Cref{order stats lem} (i), we have that
	\begin{align*}
		\P[B] = Q_k(1,v) \asymp \frac{v-k+1}{v}\bigg(1+\frac{1}{v}\bigg)^{k-1} &\asymp \frac{v-k+1}{v} \cdot e^{(k-1) \log(1+1/v)}\\
		&\asymp \frac{v-k+1}{v},
	\end{align*}
	since $(1+1/v)^{k-1} \asymp 1$ uniformly for all $1 \leq k \leq v$. Together  with the definition of $\P[B_j]$, this implies that
	$$\P\big[ B_j | B\big] = \frac{\P[B_j]}{\P[B]} \asymp \frac{v}{v-k+1} \P[B_j].$$
	Using the above estimate for $\P[B_j]$, we thus have
	\begin{align}
		\sum_{j = 1}^{k}\P\big[B_j | B\big]  & \ll \frac{1}{v^k} \sum_{j = 1}^{k} \int_{-1}^{Y_j}\binom{k+1}{j}(v-y-j)^{k-j}(y+j)^{j-1} (y+2)^2 dy. \label{bettersmirnov eq 1}
	\end{align}
	We require a bound on the terms inside the above integral. Note that we have
	$$(v-y-j)^{k-j}(y+j)^{j-1} \leq \sup_{a \in [0,v]}a^{k-j}(v-a)^{j-1} = v^{k-1}\frac{(k-j)^{k-j}(j-1)^{j-1}}{(k-1)^{k-1}}.$$
	Therefore, Stirling's approximation gives
	$$\binom{k+1}{j}(v-y-j)^{k-j}(y+j)^{j-1} \ll \frac{v^k \cdot k^{3/2}}{j^{3/2}(k-j+1)^{3/2}}.$$
	We can interchange the sum with the integral in \eqref{bettersmirnov eq 1} and then apply this bound. Note that the condition $y \leq M_j -C$ can also be written as $(y+C)^{1/\mu} \leq j \leq k - (y+C)^{1/\mu}$. As such,
	\begin{align*}
		\sum_{j = 1}^k \P[B_j|B] \ll \int_{-1}^{Y} (y+2)^2 \sum_{j = (y+C)^{1/\mu}}^{k-(y+C)^{1/\mu}} \frac{ k^{3/2}}{j^{3/2}(k-j+1)^{3/2}}  dy &\ll \int_{-1}^{Y}\frac{(y+2)^{2}}{(y+C)^{\frac{1}{2\mu}}}dy,
	\end{align*}
	where $Y = \max_{1 \leq j \leq k}Y_j.$ When $y \leq C$, note that $(y+C)^{1/(2\mu)} \asymp C^{1/(2\mu)},$ so
	$$ \int_{-1}^{C}\frac{(y+2)^{2}}{(y+C)^{\frac{1}{2\mu}}}dy \ll  C^{-\frac{1}{2\mu}}\int_{-1}^{C}(y+2)^{2}dy \ll C^{3-\frac{1}{2\mu}}.$$ 
	On the other hand, when $y > C,$ we can use the bounds $(y+C) \asymp y \asymp y+2$  to get 
	$$ \int_{C}^{Y}\frac{(y+2)^{2}}{(y+C)^{\frac{1}{2\mu}}}dy \ll  \int_{C}^{Y}y^{2-\frac{1}{2\mu}}dy \ll C^{3-\frac{1}{2\mu}}.$$
	Using the definition of $\mu$, our conclusion is that 
	$$\sum_{j = 1}^k \P[B_j|B] \ll C^{3-\frac{1}{2\mu}} = C^{-18 \epsilon +O(\epsilon^2)} \ll C^{-\epsilon}.$$
	Substituting this bound in \eqref{bettersmirnov eq 0}, we infer that $\P[\mathscr{B} | B] = 1 + O\big(C^{-\epsilon}\big).$ From this and the fact that $\mathscr{B} \subseteq B$, we deduce that
	$$\P[\mathscr{B}] = \P[B]\big(1 + O\big(C^{-\epsilon}\big)\big).$$
	This concludes the proof.
\end{proof}
	
	\section{Proof of Lower Bound Lemmas}\label{sec lower bound proofs}
	\begin{proof}[Proof of \Cref{LB initial reduction}]
		We restrict our attention to squarefree integers $n \leq x$ which can be written as $n=apb$, with $p$ prime, and $a,p,b$ satisfying:
		\begin{enumerate}[label = (\roman*)]
			\item $a \leq y^{1/8}$;
			\item For all primes $q|b,$ $q \in (y^{1/4},y^{3/4}] \cup (2y,x]$;
			\item $\log(y/p) \in \mathscr{L}(a)$.
		\end{enumerate}
		\begin{remark}
			These three conditions ensure that the representation $n = apb$ is unique: Condition (iii) implies the existence of a $d|a$ for which $y/p < d$. Hence, $y/p \leq a \leq y^{1/8}$ by condition (i). Similarly, from condition (iii) one can conclude that $p \leq 2y/d$ for some $d|a$, and so $p \leq 2y$. We end up with $y^{7/8} \leq p \leq 2y.$ The uniqueness of the representation follows easily by examining the prime factors of $n$ that lie in each of the sets $[2,y^{1/8}]$, $(y^{1/4},y^{3/4}] \cup (2y,x]$, and $(y^{7/8},2y]$. 
		\end{remark}
		The condition (iii) ensures that each $apb$ is an integer counted by $H_Q(x,y,2y)$, because if $\log(y/p) \in \mathscr{L}(a),$ then there exists $d|a$ such that $y < dp \leq 2y$ and hence $apb$ has a divisor in the desired interval for any given integer $b \leq x/ap$.
		
		For a given $a,p$, we give a lower bound on the number of $b$ such that $apb \leq x$ and conditions (i)-(iii) are satisfied. When $y^{1/2} < \frac{x}{ap} \leq 4y$, then count only those $b$'s for which $P^{-}(b) > y^{1/4}$. By \Cref{count s P-}, the number of such $b$ is $$\gg_{c,\kappa} \frac{x}{ap (\log x/ap)^{1-\density}(\log y)^{\density}} \gg \frac{x}{ap (\log x)^{1-\density}(\log y)^{\density}}.$$ 
		When $\frac{x}{ap} > 4y,$ count only those $b$'s for which $P^{-}(b) > 2y$. By \Cref{count s P-}, the number of such $b$ is $$\gg_{c,\kappa} \frac{x}{ap (\log x/ap)^{1-\density}(\log y)^{\density}} \gg \frac{x}{ap (\log x)^{1-\density}(\log y)^{\density}}.$$
		
		As such,
		\begin{equation} \label{LB eq 1}
			H_Q(x,y,2y) \gg_{c,\kappa} \frac{x}{(\log x)^{1-\delta}(\log y)^\delta} \sum_{\substack{a \leq y^{1/8} \\ a \in \S \\ \mu^2(a) = 1}} \frac{1}{a} \sum_{\substack{p \in \S \\ \log(y/p)\in \L(a)}} \frac{1}{p}.
		\end{equation}
		We now deal with the sum over $p$. If we write $\L(a)$ as a finite disjoint union of intervals: $$\L(a) = \bigcup_{i = 1}^N (u_i,v_i],$$ then within each interval, the condition $\log(y/p) \in (u_i,v_i]$ is equivalent to $p \in [ye^{-v_i},ye^{-u_i}).$ Using \Cref{defining Q}, one has
		\begin{equation}\label{LB eq 2}
			\sum_{\substack{p \in \S \\ \log(y/p) \in \L(a)}} \frac{1}{p} \asymp_\kappa \density \sum_{i = 1}^{N}\frac{(v_i-u_i)}{\log y} \geq  c\frac{L(a)}{\log y}.
		\end{equation}
		Combining \eqref{LB eq 1} and \eqref{LB eq 2} completes the proof.
	\end{proof}
	
	\begin{proof}[Proof of \Cref{LB tau}]
		For each $a \in \mathscr{A}(\bm{b})$, we have $\omega(a) = k := b_1 + \cdots +b_h$ and $\mu^2(a) = 1$. As such, $\tau(a) = 2^k$. We are left to show that
		$$\sum_{a \in \mathscr{A}(\bm{b}) } \frac{1}{a} \gg \frac{\density \log 2}{b_M! \cdots b_h!}.$$
		By definition, each interval $D_j$ contains $b_j$ primes, and so it is plain to see that
		\begin{equation}\label{eq LB tau 1}
			\sum_{a \in \mathscr{A}(\bm{b}) \cap \S} \frac{1}{a} = \prod_{j = M}^h \frac{1}{b_j!}\bigg(\sum_{p_1 \in D_j}\frac{1}{p_1} \sum_{\substack{p_2 \in D_j \\ p_2 \ne p_1}}\frac{1}{p_2} \cdots \sum_{\substack{p_{b_j}\in D_j \\ p_{b_j} \ne p_i\\(1 \leq i \leq b_j-1)}}\frac{1}{p_{b_j}}\bigg).
		\end{equation}
		We have
		\begin{align*}
			\sum_{\substack{p_{r}\in D_j \\ p_{r} \ne p_i\\(1 \leq i \leq r-1)}}\frac{1}{p_{r}} &= \sum_{\substack{p_{r}\in D_j}}\frac{1}{p_{r}} - \sum_{i = 1}^{r-1}\frac{1}{p_{i}}\\
			&\geq \density \log 2 - \frac{b_j}{\lambda_{j-1}},
		\end{align*}
		by \Cref{defining Q}. Returning to \eqref{eq LB tau 1}, we have
		
		\begin{align*}
		\sum_{a \in \mathscr{A}(\bm{b}) } \frac{1}{a} &\geq \prod_{j = M}^h \frac{1}{b_j!}\bigg(\density \log 2 - \frac{b_j}{\lambda_{j-1}}\bigg)^{b_j}.
		\end{align*}
		By \Cref{lem: size of lambda}, we know there exists a constant $R=R(c,\kappa)$ depending only on $c,\kappa$ such that
		$$\exp{[2^{j-R}]} \leq \lambda_j \leq \exp{[2^{j+R}]}.$$
		 Combining this with the assumption $b_j \leq M j$, we have 
		 \begin{align*}
		 	\prod_{j = M}^h \frac{1}{b_j!}\bigg(\density \log 2 - \frac{b_j}{\lambda_{j-1}}\bigg)^{b_j} &\gg \frac{(\density \log 2)^{b_M + \cdots +b_h}}{b_M! \cdots b_h !} \prod_{j = M}^h \bigg(1 - \frac{M j}{\log 2 \exp[2^{j-1-R}]}\bigg)^{M j}\\
		 	&\gg_{c,\kappa} \frac{(\density \log 2)^{k}}{b_M! \cdots b_h !},
		 \end{align*}
		 since $k = \sum_{j = M}^h b_j$.
	\end{proof}
	
	\begin{proof}[Proof of \Cref{LB W}]
		Write $a = p_1 \cdots p_k$ and $k_j = \sum_{i \leq j} b_i$ for $M \leq j \leq h$. Suppose 
		\begin{equation}\label{eq 1 lb w}
			p_i \in D_j \quad i \in (k_{j-1},k_j].
		\end{equation}
		By the definition of the function $W$,  
		$$W(p_1 \cdots p_k) = \#\big\{(d,d') : d,d'|p_1\dots p_k, |\log (d/d')| \leq \log 2\big\}.$$
		Thus, $W(p_1 \cdots p_k)$ is the number of ordered pairs of subsets of $\{1,\dots,k\}$ for which
		\begin{equation}\label{eq 2 lb w}
			\Big|\sum_{i \in Y}\log p_i - \sum_{i \in Z} \log p_i\Big| \leq \log 2
		\end{equation}
		
		Notice that 
		\begin{equation} \label{eq 3 lb w}
			\frac{W(a)}{a} \leq \frac{1}{b_M! \cdots b_h!} \sum_{Y,Z \subseteq \{1,\dots,k\}}  \sum_{\substack{p_1,\dots,p_k\\ \eqref{eq 1 lb w}, \eqref{eq 2 lb w}}} \frac{1}{p_1\cdots p_k}.
		\end{equation}
	
	Let us first treat the diagonal terms $Y = Z$. For such terms, \eqref{eq 2 lb w} is satisfied for all possible subsets $Y$. The inner sum of \eqref{eq 3 lb w} is thus $\ll  (\density \log 2)^k$ by repeated use of \Cref{defining Q}. As such, we have
	\begin{align} \label{eq 3.5 lb w}
		\sum_{\substack{Y,Z \subseteq \{1,\dots,k\} \\ Y = Z}}  \sum_{\substack{p_1,\dots,p_k\\ \eqref{eq 1 lb w}, \eqref{eq 2 lb w}}} \frac{1}{p_1\cdots p_k} & \ll (\density \log 2)^k\sum_{Y \subseteq \{1,\dots,k\}} 1 \\
		&= (2\density \log 2)^k. \nonumber
	\end{align}
	When $Y \ne Z$, the trick is to first fix the value of $I := \max (Y \Delta Z);$ the maximal element in the symmetric difference of $Y$ and $Z$. Suppose without loss of generality that $I \in Y \setminus Z$. The condition \eqref{eq 2 lb w} implies that
	$$\Big|\log p_I + \sum_{\substack{i \in Y \setminus Z \\ i \ne I}} \log p_i - \sum_{i \in Y \setminus Z} \log p_i \Big|,$$
	so that $U \leq p_I \leq 4U,$ where $U$ depends only on $\{p_i\}_{i \ne I}$, and is defined by
	$$U := \frac{1}{2}\exp\Big[\sum_{i \in Y \setminus Z} \log p_i - \sum_{\substack{i \in Y \setminus Z \\ i \ne I}} \log p_i\Big].$$
	Although we have made $U$ explicit, the precise definition is immaterial for this proof. If we define $E(I)$ to be the unique integer such that $p_I \in D_{E(I)}$, then by \Cref{defining Q}, we have
	\begin{equation}\label{eq 4 lb w}
		\sum_{\substack{U \leq p_I \leq 4U \\ p_I \in D_{E(I)}}} \frac{1}{p_I} \ll \frac{1}{\max(\log U, \log \lambda_{E(I)-1})} \ll_{c,\kappa} 2^{-E(I)}.
	\end{equation}
	The bound is independent of $U$. Applying \Cref{defining Q} and \eqref{eq 4 lb w} shows that
	\begin{align} \label{eq 5 lb w}
		\sum_{\substack{Y,Z \subseteq \{1,\dots,k\} \\ \max(Y \Delta Z) = I}}  \sum_{\substack{p_1,\dots,p_k\\ \eqref{eq 1 lb w}, \eqref{eq 2 lb w}}} \frac{1}{p_1\cdots p_k} &\ll_{c,\kappa} 2^{-E(I)} (\density \log 2)^k \sum_{\substack{Y,Z \subseteq \{1,\dots,k\} \\ \max(Y \Delta Z) = I}} 1\\ &= 2^{I-E(I)} (2\density \log 2)^k \nonumber
	\end{align}
	The last line above comes from the fact that there are $2^{k+I-1}$ pairs $(Y,Z)$ such that $\max(Y \Delta Z) = I$. Indeed, for each $i < I$, we have the $4$ possibilities $i \in Y \setminus Z, i \not \in Z \setminus Y, i \in Y \cap Z, i \in (Y \cup Z)^c$. For each $i > I$, we have the $2$ possibilities $i \in Y \cap Z, i \in (Y \cup Z)^c$. And so there are $4^{I-1}2^{k-I+1} = 2^{k+I-1}$ pairs.
	
	Combining \eqref{eq 3.5 lb w} and \eqref{eq 5 lb w} gives
	\begin{equation}
		\sum_{Y,Z \subseteq \{1,\dots,k\}}  \sum_{\substack{p_1,\dots,p_k\\ \eqref{eq 1 lb w}, \eqref{eq 2 lb w}}} \frac{1}{p_1\cdots p_k} \ll_{c,\kappa}(2\density \log 2)^k\Big(1+ \sum_{I = 1}^k 2^{I-E(I)}\Big).
	\end{equation}
	To finish, note that from \eqref{eq 1 lb w}, we have $E(I) = j$ for all $I \in (k_{j-1},k_j]$, and so
	$$\sum_{I = 1}^k 2^{I-E(I)} = \sum_{j = M}^h 2^{-j} \sum_{I \in (k_{j-1},k_j]} 2^{I} \ll \sum_{j = M}^h 2^{k_j-j}$$

\end{proof}

\begin{proof}[Proof of \Cref{lem from L to Yk}]
	Let $h = \Tilde{v}+M-1$. Suppose that $\bm{b} = (b_1,\dots,b_h)$ is such that $b_1 +\cdots + b_h = k$ and $b_j = 0$ for $j < M$. Further assume that
	\begin{equation}\label{eq 1 lem 3.5}
		b_j \leq 2^{j/10}\text{ and } b_{h+1-j} \leq 2^{(M+j)/10} \text{ for } 1 \leq j \leq h.
	\end{equation}
	We claim that any integer $a \in \mathscr{A}(\bm{b})$ with $\bm{b}$ as above satisfies $a \leq y^{1/8}$ as long as $M$ is sufficiently large. To see this, note that by \Cref{lem: size of lambda}, there exists $R = R(c,\kappa) > 0$ such that 
	$$\log a = \sum_{p|a} \log p = \sum_{j=1}^h \sum_{\substack{p|a\\ p \in D_j}} \log p \le \sum_{j=1}^h b_j 2^{j+R} ,$$
	as there are $b_j$ primes dividing $a$ in each interval $D_j$, and each prime $p \in D_j$ satisfies $ \log p \leq  \exp[2^{j+R}]$. By the assumption \eqref{eq 1 lem 3.5}, we have 
	\begin{align*}
		\sum_{j=1}^h b_j 2^{j+R} &\leq \sum_{j = 1}^h 2^{(9j+M+h+1)/10+R}\\
		&\leq 2^{(h+M+1)/10 + R} \sum_{j = 1}^h2^{9(j-h)/10}\\
		&\leq 2^{h+M/10+R+2}\\ & \leq \frac{\log y}{2^{9M/10 - R -2}}.
	\end{align*}
	By choosing $M$ sufficiently large in terms of $R$ (hence large in terms of $c,\kappa$), we end up with $\log a \leq (\log y)/8$, hence $a \leq y^{1/8}$. 
	
	At this point, we relabel our $b_i$'s, writing $g_i = b_{M+i-1},\;(1 \leq i \leq \Tilde{v})$. Suppose that 
	\begin{equation}\label{eq 2 lem 3.5}
		g_i \leq M+i^2\text{ and } g_{\Tilde{v}-i+1} \leq M+i^2 \text{ for } 1 \leq i \leq \Tilde{v}.
	\end{equation}
	This condition implies condition \eqref{eq 1 lem 3.5} provided that $M$ is large. Write $$G_j = \sum_{i = 1}^j g_i = \sum_{i = M}^{M+j}b_i,$$ Let $C$ be a sufficiently large absolute constant, and suppose that
	\begin{equation}\label{eq 3 lem 3.5}
		j - G_j \geq \max \big\{-1,(\min\{G_j,k-G_j\})^{1/7}-C\big\} \quad (1 \leq j \leq \Tilde{v}).
	\end{equation}
	For $b_j$'s chosen as described, we group the terms of our sum by the value of $m = G_i$. Since $G_i$ is non-decreasing and takes integer values between $0$ and $k$, we have
	\begin{align*}
		\sum_{j =M}^h 2^{-j +b_M+\cdots +b_j} &= 2^{1-M}\sum_{i=1}^{\Tilde{v}} 2^{-(i -G_{i})} \\
		&= 2^{1-M} \sum_{m=0}^k \sum_{\substack{1 \leq i \leq \Tilde{v} \\ G_i = m}} 2^{-(i-m)}.
	\end{align*}
	For each $m \in \{0, \dots, k\}$, let $I_m = \{i \in [1,\Tilde{v}] : G_i = m\}$. If $I_m$ is non-empty, let $i_m$ be its smallest index. As $i$ increases within $I_m$, the difference $i-m$ increases by $1$ at each step. Thus, the inner sum is bounded by a geometric series:
	$$ \sum_{i \in I_m} 2^{-(i-m)} \leq \sum_{\ell = 0}^\infty 2^{-(i_m - m + \ell)} = 2 \cdot 2^{-(i_m - m)}. $$
	By \eqref{eq 3 lem 3.5} evaluated at the starting index $i_m$, we know $i_m - m \geq (\min\{m,k-m\})^{1/7} - C$. Therefore,
	$$ 2 \cdot 2^{-(i_m - m)} \leq 2^{C+1} 2^{-(\min\{m,k-m\})^{1/7}}. $$
	Summing this bound over all possible values of $m$, we obtain
	\begin{align*}
		\sum_{j =M}^h 2^{-j +b_M+\cdots +b_j} &\leq 2^{2-M+C} \sum_{m=0}^k 2^{-(\min\{m,k-m\})^{1/7}} \\
		&\leq 2^{3-M+C} \sum_{m=0}^{\lfloor k/2 \rfloor} 2^{-m^{1/7}}.
	\end{align*}
	Finally, by comparing with an integral, we see that 
	$$\sum_{m=0}^{\lfloor k/2 \rfloor} 2^{-m^{1/7}} \leq \int_0^\infty 2^{-t^{1/7}}dt \leq 2^{20}.$$
	Hence, from \eqref{L to factorials} we conclude
	\begin{equation*}
		\sum_{\substack{a \in \mathscr{A}(\bm{b}) \\ \omega(a) = k}} \frac{L(a)}{a} \gg_{c,\kappa}  \frac{(2 \density \log 2)^k}{b_M! \cdots b_h!} \bigg[ 1 + \sum_{j =M}^h 2^{-j +b_M + \cdots +b_j} \bigg]^{-1} \ge \frac{1}{1+2^{23+C-M}} \cdot \frac{(2 \density \log 2)^k}{g_1! \cdots g_{\Tilde{v}}!}.
	\end{equation*}
	Since $M$ depends only on $c,\kappa$, we can rewrite this as
	\begin{equation}\label{L to factorials 2}
		\sum_{\substack{a \in \mathscr{A}(\bm{b}) \\ \omega(a) = k}} \frac{L(a)}{a} \gg_{c,\kappa} 2^{-C} \cdot \frac{(2 \density \log 2)^k}{g_1! \cdots g_{\Tilde{v}}!}.
	\end{equation}
	The quantity $(g_1! \cdots g_{\Tilde{v}}!)^{-1}$ is naturally interpreted as the area of a geometric region in $\R^k$. For $\bm{g} = (g_1, \dots, g_{\Tilde{v}})$. define
	$$\mathcal{R}(\bm{g}) = \{x \in \R^k : 0 \leq x_1 \leq \cdots \leq x_k <v , x_{G_{i-1}+1},\dots,x_{G_{i}} \in [i-1,i)\;1 \leq i \leq \Tilde{v}\},$$
	with the convention that $G_0 = 0$. We then have
	\begin{equation} \label{eq 3.5 lem 3.5}
		\frac{1}{g_1! \cdots g_{\Tilde{v}}!} = \Vol(\mathcal{R}(\bm{g})).
	\end{equation}
	Let $$\mathscr{G} = \Big\{\bm{g}=(g_1, \dots, g_{\Tilde{v}}): \sum_{i = 1}^{\Tilde{v}} g_i =k, \text{ and } \eqref{eq 2 lem 3.5}, \eqref{eq 3 lem 3.5}\text{ hold} \Big\}.$$
	We will show that
	\begin{equation}\label{eq 4 lem 3.5}
		\Tilde{v} \mathcal{Y}_k (\Tilde{v},C) \subseteq \bigcup_{\bm{g} \in \mathscr{G}} \mathcal{R}(\bm{g}).
	\end{equation}
	Given $\xi_i \in \mathcal{Y}_k(\Tilde{v},C)$, we set $x_i = \Tilde{v}\xi_i$. We claim that $x_i \in \mathcal{R}(\bm{g}),$ where $g_i$ is defined to be the number of $x_j \in [i-1,i)$. Let us verify that \eqref{eq 3 lem 3.5} holds for such an $x_i$. By the definition of $g_i$, we have $j > x_{G_j}$ for each $j$. Combining this with property (iii) of $\mathcal{Y}_k(\Tilde{v},C)$, we see that
	$$j > x_{G_j} \geq \max\{G_j-1, G_j + (\min\{G_j,k-G_j\})^{1/7} -C\}.$$
	Hence, $j - G_j \geq \max\{-1,(\min\{G_j,k-G_j\})^{1/7}-C\},$
	and so \eqref{eq 3 lem 3.5} is satisfied. Similarly, by condition (ii) in the definition of $\mathcal{Y}_k(v',C)$, we have $x_{M+i^2} > i,$ so $g_i \leq M+i^2$ for $(1 \leq i \leq \sqrt{k-M})$. One can similarly show by using the condition $x_{k-(M+i^2)+1} < \Tilde{v}-i$ that $g_{\Tilde{v}-i+1} \leq M+i^2.$ Hence, \eqref{eq 2 lem 3.5} is satisfied. 
	
	 This shows that $\bm{x} = (x_1,\dots, x_k) \in \mathcal{R}(\bm{g})$ for a certain $\bm{g} \in \mathscr{G}$, and so \eqref{eq 4 lem 3.5} holds. As such, we have
	\begin{equation}\label{eq 5 lem 3.5}
		\Vol\Big(\bigcup_{\bm{g} \in \mathscr{G}} \mathcal{R}(\bm{g})\Big) \ge \Tilde{v}^k \Vol\Big( \mathcal{Y}_k(\Tilde{v},C)\Big).
	\end{equation}
	Combining \eqref{L to factorials 2}, \eqref{eq 3.5 lem 3.5} and \eqref{eq 5 lem 3.5}, we have proven that
	$$\sum_{\substack{a \in \mathscr{A}(\bm{b}) \\ \omega(a) = k}} \frac{L(a)}{a} \gg_{c,\kappa} (2 \Tilde{v} \density \log 2)^k\frac{\Vol\big( \mathcal{Y}_k(\Tilde{v},C)\big)}{2^C}.$$
	In order to complete the proof, we must replace $\Tilde{v}$ with $v$ in our final estimate. To do so, simply note that 
	$$\Tilde{v}^k = v^k (1-(2M)/v)^k \geq e^{-2M} \gg_M v^k.$$ 
	Since $M$ depends entirely on $c,\kappa,$ we can replace $\gg_M$ with $\gg_{c,\kappa}$. This completes the proof.
\end{proof}
	\begin{proof}[Proof of \Cref{lem LB Yk}]
		To begin, recall the definition of $S_k(u,v)$ from the proof of \Cref{lemma better smirnov}:
		$$S_k(u,v) := \Big\{\bm{\xi} \in \mathbb{R}^k : 0 \leq \xi_1 \leq \cdots \leq \xi_k \leq 1 ; \xi_i \geq \frac{i-u}{v}\Big\}.$$
		By \Cref{order stats lem} (i), one has
		\begin{equation}\label{apply wool thm Sk}
			\Vol(S_k(1,\Tilde{v})) = \frac{1}{k!} \frac{\Tilde{v}-k+1}{\Tilde{v}}\bigg(1+\frac{1}{\Tilde{v}}\bigg)^{k-1} \asymp \frac{\Tilde{v}-k+1}{\Tilde{v} \cdot k!}.
		\end{equation}
		The desired result will follow once we prove that $\Vol(\mathcal{Y}_k(\Tilde{v},C)) \gg \Vol(S_k(1,\Tilde{v})).$ 
		
		Let $\mathscr{B}(C)$ denote the geometric region in $\R^k$ corresponding to the event $\mathscr{B}$ defined in \Cref{lemma better smirnov}, evaluated with parameters $v = \Tilde{v}$ and $\epsilon = 1/6 - 1/7 = 1/42$. By definition, $\mathscr{B}(C)$ is exactly the set of points satisfying conditions (i) and (iii) of $\mathcal{Y}_k(\Tilde{v},C)$. By \Cref{lemma better smirnov}, we have 
		\begin{equation} \label{ihatethis3}
			\Vol(\mathscr{B}(C)) = \Vol(S_k(1,\Tilde{v})) \big(1 - O(C^{-1/42})\big).
		\end{equation}
		
		To account for condition (ii), we define the exceptional sets
		\begin{align*}
			V_1(a,b) &:= \Big\{\bm{\xi} \in \mathscr{B}(C) : \xi_a \leq \frac{b}{\Tilde{v}}\Big\}, \\
			V_2(a,b) &:=  \Big\{\bm{\xi} \in \mathscr{B}(C) : \xi_{k+1-a} \geq 1-\frac{b}{\Tilde{v}}\Big\}.
		\end{align*}
		Since $\mathscr{B}(C) \subseteq S_k(C, \Tilde{v})$, we can bound the volume of these sets exactly as in the proof of Lemma 4.9 of \cite{ford2008distribution}, substituting $C$ for the number $u_0$ which appears in the mentioned reference. This yields
		\begin{equation}
			\label{ihatethis2} \sum_{j = 1}^2 \sum_{1 \leq i \leq \sqrt{k-M}} \Vol\big(V_j(M+i^2,i)\big) \leq \frac{K_1 C}{2^M} \Vol(S_k(1,\Tilde{v})),
		\end{equation}
		for some absolute constant $K_1 > 0$. Since $\mathcal{Y}_k(\Tilde{v},C)$ contains all points in $\mathscr{B}(C)$ except those in the sets $V_j$, we have
		$$\Vol(\mathcal{Y}_k(\Tilde{v},C)) \geq \Vol(S_k(1,\Tilde{v})) \bigg(1 - O(C^{-1/42}) - \frac{K_1 C}{2^M}\bigg).$$
		
		By taking $C$ to be a sufficiently large absolute constant, we make the $O(C^{-1/42})$ term $\leq 1/4$. We may then ensure that $M$ was chosen sufficiently large in terms of $C$ so that $K_1 C 2^{-M} \leq 1/4$. This forces
		$$\Vol(\mathcal{Y}_k(\Tilde{v},C)) \geq \frac{1}{2}\Vol(S_k(1,\Tilde{v})) \gg \frac{\Tilde{v}-k+1}{\Tilde{v} \cdot k!}.$$
		To complete the proof, simply note that $\Tilde{v} = v +O(M),$ and so by taking $y$ sufficiently large in terms of $M$ (hence, in terms of $c,\kappa$ only), we have
		$$\frac{\Tilde{v}-k+1}{\Tilde{v} \cdot k!} \gg_M \frac{v-k+1}{v \cdot k!},$$
		which concludes the proof.
	\end{proof}
		\section{Proof of Upper Bound Lemmas}\label{sec upper bound proofs}
	\begin{proof}[Proof of \Cref{lem UB 1}]
		For each $n \leq x$ satisfying $\tau(n,y,2y) \geq 1$, write $n = n'n''$ where $n'$ is squarefree, $n''$ is squarefull, and $(n',n'') = 1$. At the cost of a small error, we can assume $n'' \leq y^{1/10}.$ Indeed, 
		$$\#\{n'n'' \in \S \cap [1,x] : n'' > y^{1/10} \} \ll \frac{x}{y^{1/20}}.$$
		As such, we have
		\begin{equation}\label{upper bound lemma 1 eq 1}
			H_Q(x,y,2y) = \sum_{\substack{n'n'' \in \S \cap [1,x]\\ n'' \leq y^{1/10}\\ \tau(n'n'',y,2y) \geq 1}}1 +O\left(\frac{x}{y^{1/20}}\right).
		\end{equation}
		Whenever $\tau(n'n'',y,2y) \geq 1$, then there exists $g$ dividing $n''$ such that $n'$ has a divisor in $(y/g,2y/g]$. Indeed, if $n'n'' = dm$ with $d \in (y,2y]$, then define
		$$f := (m,n''); \quad g := \frac{n''}{f},$$
		so that $(g,f) = (g,m) = 1$. Clearly $g|n''$. Furthermore, $g | d$ and $f|m$. We can then write $n' = (d/g)(m/f),$ and we see that $d/g$ is an integer dividing $n'$ which lies in $(y/g,2y/g]$. Returning to \eqref{upper bound lemma 1 eq 1}, we have
		\begin{equation}\label{upper bound lemma 1 eq 2}
			H_Q(x,y,2y) \leq \sum_{n'' \in \S \cap [1, y^{1/10}]}\sum_{\substack{g|n''}}H^*_Q \bigg(\frac{x}{n''},\frac{y}{g},\frac{2y}{g}\bigg) +O\left(\frac{x}{y^{1/20}}\right),
		\end{equation}
		where $H^*_Q(x,y,2y)$ denotes \textit{squarefree} integers not exceeding $x$ having a divisor in $(y,2y]$. Write $(y_1,z_1,x_1) = (y/g,2y/g,x/n''),$ where $g$ is some divisor of $n''$. We now want to show that for $y_1 \leq x_1 ^{5/9}$ and $y_1/2 \leq z_1 \leq x_1$, we have 
		\begin{equation}\label{upper bound lemma 1 eq 3}
			H^*_Q(x_1,y_1,z_1) - H^*_Q(x_1/2,y_1,z_1) \ll \frac{x_1}{(\log x_1)^{1-\density}} \sum_{a \in \mathscr{P}(2y) \cap S}\frac{L(a)}{a\log(y^{1/6}/a + P^+(a))^{1+\density}} 
		\end{equation}
		Write $y_2 = \frac{x_1}{2z_1}, z_2 = x_1/y_1$. Given a squarefree $n \in (x_1/2,x_1]$ for which $\tau(n,y_1,z_1)\geq 1$, let us write $n = m_1m_2$ with $y_i \leq m_i \leq z_i$. There exists $j' \in \{1,2\}$ such that $P^+(m_{j'})<P^+(m_{3-j'})$.\footnote{Here the squarefree assumption comes in handy.} We set $p = P^+(m_j')$. Write $n = apb$, with $P^+(a) < p < P^-(b)$ and $b > p$. From now on, write $j' = j$ for ease of notation. We claim that $\tau(ap,y_{j},z_j) \geq 1.$ To see this, suppose $j = 1$. Note that $b|m_2$, as all prime factors of $b$ are larger than $p$, and thus cannot divide $m_1$, which has largest prime factor smaller than $p$. So $m_1|ap$, and $m_1 \in (y_1,z_1]$. The case when $j=2$ is analogous. 
		
		Since $\tau(ap,y_{j},z_j) \geq 1$, then $p \geq y_j/a.$ We claim that $y_j \geq y_1^{1/6}$. Indeed, $y_2 = x_1/2z_1.$ So by our assumptions,
		$$y_2 = \frac{x g}{4n'' y} \geq \frac{y}{4n''} \geq y^{1/6}.$$
		This proves that $p \geq \max(y^{1/5}/a,P^+(a))$. As such, by \Cref{count s P-} the number of $b$ we count is
		$$\ll \frac{x_1}{ap (\log p)^\density (\log(x_1/ap))^{1-\density}} \ll \frac{x_1}{ap (\log x_1)^{1-\density} \log \max(y^{1/6}/a,P^+(a))^\density},$$
		Since $\tau(a,y_j/p,z_j/p) \geq 1$, then $\log(y_j/p) \in \mathscr{L}(a)$. By \Cref{defining Q}, we have
		$$\sum_{\substack{\log(y_j/p) \in \mathscr{L}(a) \\ p \geq P^+(a)}} \frac{1}{p} \ll  \frac{\density L(a)}{\log \max(y^{1/6}/a,P^+(a))}.$$
		\Cref{upper bound lemma 1 eq 3} follows. We can now apply this to finish the proof of our lemma. We must bound $H^*_Q \big(\frac{x}{n''},\frac{y}{g},\frac{2y}{g}\big)$ for $g|n'' \leq y^{1/10}$. To accomplish this, write $$(0,x/n''] = (0,x/(n'' y^{1/10})] \cup \bigcup_{r = 0}^{\frac{\log y}{10 \log 2}} (x/(2^{r+1}n''),(x/(2^{r}n'')].$$
		On the interval $(0,x/(n'' y^{1/10})]$ we include all integers in our $H_Q^*$ count. On each of the dyadic intervals $(x/(2^{r+1}n''),(x/(2^{r}n'')]$, we apply \eqref{upper bound lemma 1 eq 3}. This leads to the upper bound
		\begin{align*}
			&H^*_Q \bigg(\frac{x}{n''},\frac{y}{g},\frac{2y}{g}\bigg) \ll \\ &\ll \frac{x}{n'' y^{1/10} } + \sum_{r \leq \frac{\log y}{10 \log 2}}\frac{x}{2^r n'' (\log (x/2^r n'')^{1-\density}}\sum_{a \in \mathscr{P}(2y) \cap \S}\frac{L(a)}{a \log((y/g)^{1/6}/a+P^+(a))^{1+\density}}\\
			&\ll \frac{x}{n'' y^{1/10} } + \frac{x}{n'' (\log x)^{1-\density}}\sum_{a \in \mathscr{P}(2y) \cap \S}\frac{L(a)}{a \log((y/g)^{1/6}/a+P^+(a))^{1+\density}}\\
			&\ll \frac{x}{n'' y^{1/10}} + \frac{x}{n'' (\log x)^{1-\density}}\sum_{a \in \mathscr{P}(2y) \cap \S}\frac{L(a)}{a \log(y^{3/20}/a+P^+(a))^{1+\density}}.
		\end{align*}
		We always have $1 \in \mathscr{P}(2y) \cap \S$, and so the above $a$ sum is $$ \gg x L(1)/n'' = x (\log 2)/n'' \gg \frac{x}{n'' \log y} \gg \frac{x}{n''y^{1/10}}.$$ As such, we are simply left with the bound
		$$	H^*_Q \bigg(\frac{x}{n'' (\log x)^{1-\density}},\frac{y}{g},\frac{z}{g}\bigg) \ll \frac{x}{n''}\sum_{a \in \mathscr{P}(2y) \cap \S}\frac{L(a)}{a \log(y^{3/20}/a+P^+(a))^{1+\density}}.$$
		Substituting this bound into \eqref{upper bound lemma 1 eq 2} gives
		\begin{align*}
			&H_Q(x,y,2y) \ll \\ &\ll \frac{x}{(\log x)^{1-\density}}\sum_{n'' \in \S \cap [1, y^{1/10}]}\frac{1}{n''} \sum_{\substack{g|n''}}\sum_{a \in \mathscr{P}(2y) \cap \S}\frac{L(a)}{a \log(y^{3/20}/a+P^+(a) )^{1+\density}} +O\left(\frac{x}{y^{1/20}}\right)\\
			&\ll \frac{x}{(\log x)^{1-\density}}\sum_{a \in \mathscr{P}(2y) \cap \S}\frac{L(a)}{a \log(y^{3/20}/a+P^+(a) )^{1+\density}} \sum_{n'' \in \S \cap [1, y^{1/10}]}\frac{\tau(n'')}{n''} +O\left(\frac{x}{y^{1/20}}\right)\\
			&\ll \frac{x}{(\log x)^{1-\density}}\sum_{a \in \mathscr{P}(2y) \cap \S}\frac{L(a)}{a \log(y^{3/20}/a+P^+(a) )^{1+\density}}.
		\end{align*}
		This completes the proof of \Cref{lem UB 1}.\end{proof}
		\begin{proof}[Proof of \Cref{lem UB 4}]
			The proof is largely similar to the proof of Lemma 3.5 in \cite{ford2006integers}. Let us write $a =p_1 \cdots p_k$. Recalling the definition of $D_j$ and $\lambda_j$ from Section 3, let us define $j_i$ to be the unique integer for which $p_i \in D_{j_i}$. Note that $2 \leq p_i \leq 2y$ for each $i$, and so $0 \leq j_i \leq v +O(1).$ Let $\tilde{p_1} \leq \tilde{p_2} \leq \cdots \leq \tilde{p_k}$ denote the rearrangement of the $p_i's$ into increasing order, so that $\tilde{p_i} \in D_{\tilde{j_i}}$ and $\tilde{j_1} \leq \tilde{j_2} \leq \cdots \leq \tilde{j_k}$. By \Cref{lemma L properties}, we have 
			$$L(a) \leq 2^k \min_{0 \leq g \leq k}2^{-g}(\log(2\tilde{p_1}\cdots \tilde{p_g})) \leq 2^{k+K}F(\bm{j}),$$
			with $$F(\bm{j}) = \min_{0 \leq g \leq k}2^{-g}(2^{\tilde{j_1}} + \cdots + 2^{\tilde{j_g}} + 1).$$
			Thus,
			\begin{align*}
				T_Q(k,2y) &\leq \frac{2^{k+K}}{k!} \sum_{j_1,\dots, j_k = 1}^{v+O(1)}F(\bm{j}) \sum_{\substack{p_1, \dots p_k \\ p_i \in D_{j_i}}}\frac{1}{p_1 \cdots p_k}\\
				&\ll \frac{1}{k!}(2\density \log 2)^k \sum_{\bm{j}} F(\bm{j}).
			\end{align*}
			It is shown in \cite{ford2006integers}, p.11 that 
			$$\sum_{\bm{j}} F(\bm{j}) \ll (v+O(1))^k k! U_k(v).$$
			Since $k \leq 10v$, then $(v+O(1))^k \ll v^k,$ and this completes the proof of \Cref{lem UB 4}. \end{proof}
		\newpage
		\section*{Index of symbols}
		\begin{table}[h]
			\centering
			\renewcommand{\arraystretch}{1.3}
			\begin{tabular}{lp{6.5cm}l}
				\hline
				\textbf{Symbol} & \textbf{Meaning} & \textbf{Definition} \\
				\hline
				$Q$ & A fixed set of prime numbers & \eqref{definition of delta} \\
				$\density$ & The relative density of the set of primes $Q$ & \eqref{definition of delta} \\
				$\kappa$ & Parameter bounding the error term for the prime counting function of $Q$ & \eqref{definition of delta} \\
				$\S_Q$ (or $\S$) & Set of integers whose prime factors all belong to $Q$ & Section 1 \\
				$H_Q(x,y,z)$ & Number of integers $n \in \S_Q \cap [1,x]$ having a divisor $d \in (y,z]$ & Section 1 \\
				$G(\density)$ & Exponent of the logarithmic factor for the main asymptotic & \Cref{mainthm} \\
				$E(y;\density)$ & Function describing the transitional phase behavior near $\density = 1/\log 4$ & \Cref{mainthm} \\
				$A_Q(N)$ & Number of distinct integers of the form $ab$ with $a,b \in \S_Q \cap [1,N]$ & Section 1 \\
				$\lambda$ & The parameter $2\density \log \log y$ & \Cref{prop 1.4} \\
				$v$ & The integer $\lfloor\frac{1}{\log 2}\log \log y\rfloor$ & \Cref{prop 1.4} \\
				$\mathscr{L}(a)$ & Union of intervals $(\log d - \log 2, \log d]$ over divisors $d|a$ & \eqref{def L} \\
				$L(a)$ & Lebesgue measure of $\mathscr{L}(a)$ & \eqref{def L} \\
				$W(a)$ & Number of divisor pairs $(d,d')$ of $a$ with $|\log(d/d')| \leq \log 2$ & \Cref{cauchy schwarz L} \\
				$\mathcal{Y}_k(\Tilde{v},C)$ & Geometric region defining strong barrier conditions for normalized order statistics & \Cref{lem from L to Yk} \\
				$U_k(v)$ & Multivariate volume integral over the unit simplex used for upper bounds & \eqref{def Uk} \\
				$Q_k(u,v)$ & Generalized Smirnov statistic evaluating uniform order statistics under barrier constraints & \eqref{def Qk} \\
				\hline
			\end{tabular}
			\vspace{0.2cm}
			\caption{ Index of frequently used symbols (non-exhaustive).}
			\label{table:notation}
		\end{table}

\end{document}